\documentclass[a4paper,11pt]{article}
\usepackage[utf8]{inputenc}
\usepackage{amsmath}
\usepackage{amsfonts}
\usepackage{amssymb}
\usepackage{amsthm}
\usepackage[english]{babel}
\usepackage{fontenc}
\usepackage{mathtools}
\mathtoolsset{showonlyrefs}
\usepackage{dsfont}
\usepackage{enumitem}
\usepackage[hmargin=3cm,vmargin=2.5cm,paper=a4paper]{geometry}

\newcommand{\ls}{\leqslant}
\newcommand{\gs}{\geqslant}
\renewcommand{\div}{\operatorname{div}}
\newcommand{\R}{\mathbb R}
\newcommand{\Rd}{\mathcal{R}}
\newcommand{\Rda}{\mathcal{R}^\ast}
\newcommand{\A}{\mathcal{A}}
\newcommand{\per}{\operatorname{Per}}
\newcommand{\eps}{\varepsilon}
\newcommand{\BV}{\mathrm{BV}} 

\newcommand{\dd}{\, \mathrm{d}}
\newcommand{\Id}{\mathrm{Id}}
\newcommand\restr[2]{{
\left.\kern-\nulldelimiterspace #1 \vphantom{\big|} \right|_{#2} 
}}
\newcommand{\TV}[1]{\mathrm{TV}{(#1)}}
\newcommand{\TVper}[1]{\mathrm{TV}_{\mathrm{per}}{(#1)}}
\newcommand{\TVO}[1]{\mathrm{TV}{(#1\, ; \, \Omega)}}
\newcommand{\Br}{B(x,r)}

\newcommand{\abs}[1]{\lvert#1\rvert}
\newcommand{\norm}[1]{\left\|#1\right\|}
\newcommand{\set}[1]{\left\{#1\right\}}

\newcommand{\supp}{\mathop{\mathrm{supp}}}
\newcommand{\diam}{\mathop{\mathrm{diam}}}

\newcommand{\scal}[2]{ \left \langle #1, \ #2 \right \rangle}

\newtheorem{theorem}{Theorem}
\newtheorem{prop}{Proposition}
\newtheorem{lemma}{Lemma}
\theoremstyle{definition}
\newtheorem{definition}{Definition}
\newtheorem{example}{Example}
\theoremstyle{remark}
\newtheorem{remark}{Remark}
\numberwithin{equation}{section}
\usepackage{color}
\usepackage[dvipsnames]{xcolor}

\begin{document}
\title{Influence of dimension on the convergence of level-sets in total variation regularization }
\author{Jos\'e A. Iglesias\thanks{Johann Radon Institute for Computational and Applied Mathematics (RICAM), Austrian Academy of Sciences, Linz, Austria (\texttt{jose.iglesias{@}ricam.oeaw.ac.at})} , Gwenael Mercier\thanks{Computational Science Center, University of Vienna, Austria (\texttt{gwenael.mercier{@}univie.ac.at})}}
\date{}
\maketitle
\vspace{-0.2cm}
\begin{abstract}
We extend some recent results on the Hausdorff convergence of level-sets for total variation regularized linear inverse problems. Dimensions higher than two and measurements in Banach spaces are considered. We investigate the relation between the dimension and the assumed integrability of the solution that makes such an extension possible. We also give some counterexamples of practical application scenarios where the natural choice of fidelity term makes such a convergence fail.
\end{abstract}

\section{Introduction}
In a few recent papers, several results have been shown linking the source condition for convex regularization introduced in \cite{BurOsh04} to the convergence in Hausdorff distance of level-sets of total variation regularized solutions of inverse problems, as the amount of noise and the regularisation parameter vanish simultaneously. Such a mode of convergence, although seldom used, is of particular interest in the context of recovery of piecewise constant coefficients as well as in the processing of images composed mainly of objects separated by clear boundaries. In these situations, Hausdorff convergence of level-sets can be seen as uniform convergence of the geometrical objects appearing in the data.

To be more specific, in \cite{ChaDuvPeyPoo17} such a convergence is obtained for the denoising problem in the entire plane with $L^2$ fidelity term, and in \cite{IglMerSch18} the authors extend the result to bounded domains and to general linear inverse problems. These results have two common features. First, they are written in a Hilbert space framework, allowing to easily study the convergence of dual solutions. Second, the analysis is performed in the plane where this Hilbert framework corresponds to the optimal scaling where weak regularity for level-sets as well as good behavior at infinity can be proved, both of them being related to equi-integrability of these dual solutions locally or at infinity.

In \cite{BurKorRas19}, similar results are obtained in the setting of imperfect forward models, with measurements in $L^\infty$ and where an $L^1$ norm term is added to the regularization. There, it is assumed that the operators are bounded from $L^1$ in a bounded domain ($q=1$ in the notation below), a case that we do not treat since then boundedness in $(L^1)^\ast = L^\infty$ directly implies equi-integrability in $L^p$ for any $p$.

Our aim is to extend this type of result to different choice of integrability and measurements made in more general Banach spaces. We will see that this extension requires some particular choices of these ingredients, and present some positive results as well as counterexamples. 

More precisely, we study convergence, as the positive regularization parameter $\alpha$  and the noise $w$ simultaneously vanish, of level-sets of minimizers of
\begin{equation}\label{eq:primalalphaw}\inf_{u \in L^q(\Omega)} \frac{1}{\sigma}\|Au-f-w\|^\sigma_{Y}+\alpha\TV{u},\tag{$P_{\alpha,w}$}\end{equation}
with $q, \sigma>1$ and $\Omega \subseteq \R^d$, $d>1$. We assume $q \ls d/(d-1)$, which implies that the conjugate exponent $q':= q/(q-1)\gs d$. Here $A:L^q(\Omega) \to Y$ is linear bounded, where $Y$ is a locally uniformly convex Banach space, with dual $Y^\ast$ which is also assumed to be uniformly convex and with modulus of uniform convexity of power type $\tau \ls \sigma'$, where $\sigma'= \sigma/(\sigma - 1)$ (see Definition \ref{def:unifconv} and Proposition \ref{prop:unifconv}). The power $\sigma>1$ allows for natural choices of data term depending on the space $Y$, beyond the case of Hilbert space where $\sigma=2$.

\subsection{Preliminaries}
\paragraph{A few results on geometry of Banach spaces.}
We begin by making precise our requirements for the measurement space $Y$.
\begin{definition}\label{def:unifconv}
Let $\phi: Y \to \R$ be a convex function. We say that $\phi$ is \emph{locally uniformly convex} if for any $f \in Y$, there exists a nondecreasing real function $h_\phi^f >0$ such that for every $g \in Y$ with $g \neq f$ and $0 \ls t \ls 1$,
\begin{equation} \label{eq:locunifconv} \phi\left((1-t)f+tg)\right) \ls (1-t)\phi(f) + t\phi(g) - t(1-t) h^f_\phi\left(\|f-g\|_{Y}\right).\end{equation}

The function $\phi$ is called (globally) \emph{uniformly convex} \cite[Chapter 5.3]{BorVan10} if there exists a nondecreasing $h_\phi >0$ such that for all $f \neq g \in Y$ and $0 \ls t \ls 1$ we have
\begin{equation}\label{eq:globunifconv}\phi\left((1-t)f+tg)\right) \ls (1-t)\phi(f) + t\phi(g) - t(1-t)h_\phi\left(\|f-g\|_{Y}\right).\end{equation}
Furthermore, if two functions $h_\phi, \tilde h_\phi$ satisfy \eqref{eq:globunifconv}, then the function $s \mapsto \max(h_\phi(s), \tilde h_\phi(s))$ does too, so there is a largest such function that we denote by $\delta_\phi$ and call the modulus of uniform convexity of $\phi$. If $\delta_\phi(\eps) \gs C \eps^p$ for some $C>0$, $p>1$ and all $\eps \gs 0$, we say that this modulus of uniform convexity is of power type $p$.

Moreover, the function $\phi$ is said to be \emph{strictly convex} when for all $f, g \in Y$ with $f \neq g$ and $0 < t < 1$ we have
\begin{equation}\phi\left((1-t)f+tg)\right) < (1-t)\phi(f) + t\phi(g).\label{eq:stricconv}\end{equation}

Clearly, uniform convexity is stronger than local uniform convexity, which in turn implies strict convexity.
\end{definition}

The main quantitative result about uniformly convex functions that we will use is the following uniform monotonicity inequality for subgradients:
\begin{lemma}\label{lem:unifmonoton}
Let $\phi: Y \to \R$ be a convex function with modulus of uniform convexity $\delta_\phi$, and denote by $\partial \phi(f) \subset Y^\ast$ the subgradient of $\phi$ at $f$. Then, if $v_f \in \partial \phi(f)$ and $v_g \in \partial \phi(g)$ we have the uniform monotonicity inequality
\begin{equation}\label{eq:Vident}\scal{v_f - v_g}{f-g}_{(Y^\ast, Y)}\gs 2 \delta_\phi\left(\|f-g\|_{Y}\right).\end{equation}
\end{lemma}
\begin{proof}
Since $v_f \in \partial \phi(f)$ we can write for each $0 < t < 1$
\begin{equation}\begin{aligned}
\phi(f)+\scal{v_f}{t(g-f)}_{(Y^\ast, Y)} &\ls \phi\big(f + t(g-f)\big)=\phi\big((1-t)f + tg\big)\\
&\ls (1-t)\phi(f) + t\phi(g) - t(1-t)\delta_\phi(\|f-g\|_{Y}),
\end{aligned} \label{eq:unifsubg1}\end{equation}
or
\[t \scal{v_f}{g-f}_{(Y^\ast, Y)} \ls t\phi(g) - t\phi(f) - t(1-t)\delta_\phi(\|f-g\|_{Y}),\]
in which we can divide by $t$ and take the limit as $t \to 0$ to obtain
\begin{equation}\label{eq:unifsubg2}\phi(g) \gs \phi(f) + \scal{v_f}{g-f}_{(Y^\ast, Y)} + \delta_\phi(\|f-g\|_{Y}).\end{equation}
Similarly, for $v_g$ we get
\begin{equation}\label{eq:unifsubg3}\phi(f) \gs \phi(g) + \scal{v_g}{f-g}_{(Y^\ast, Y)} + \delta_\phi(\|f-g\|_{Y}),\end{equation}
and using \eqref{eq:unifsubg2} in \eqref{eq:unifsubg3} we get \eqref{eq:Vident}.
\end{proof}

The uniform convexity notions of Definition \ref{def:unifconv} give rise to analogous ones for Banach spaces through their norms \cite[Def.\ 5.3.2, Thm.\  5.2.5]{Meg98}:
\begin{definition}
A Banach space $Y$ is said locally uniformly convex (resp. (globally) uniformly convex, strictly convex) if  \eqref{eq:locunifconv} (resp. \eqref{eq:globunifconv}, \eqref{eq:stricconv}) hold for $f,g$ belonging to the unit sphere and $\phi$ is the norm of $Y$. The modulus of uniform convexity of $Y$ is the corresponding $\delta_{\| \cdot \|}$ for such points.
\end{definition}

The uniform convexity of $Y$ and $Y^\ast$ that we assume is arguably not a strong restriction, since it is satisfied by many natural spaces arising in the study of inverse problems for physical models (see \cite[Prop.\ 11.12]{Bre11} for quotients,  \cite[Thm.\ 3.9 and Thm.\ 3.12]{AdaFou03} for duals of Sobolev spaces, \cite[Example 2.47]{SchKalHofKaz12} for the power types and \cite{Han56} for the precise moduli of $L^p$).
\begin{prop}
Let $1 < p < \infty$, $p'=p/(p-1)$ and $\Omega \subseteq \R^d$ an open set.
\begin{itemize}
\item The space of sequences $\ell^p$ is uniformly convex, and in consequence so is the dual space $\ell^{p'}$.
\item The space $L^p(\Omega)$ is also uniformly convex , as is the dual $L^{p'}(\Omega)$.
\item Sobolev spaces $W^{k,p}(\Omega)$. Since they can be isometrically embedded in $L^p(\Omega;\R^N)$ for some $N$, they are uniformly convex. The representation theorem for $(W^{k,p}(\Omega))^\ast$ as a subspace of $L^{p'}(\Omega;\R^N)$ implies that it is also uniformly convex. Similarly, $W^{k,p}_0(\Omega)$ and its dual $W^{-k,p'}(\Omega)$ are also uniformly convex.
\item The modulus of uniform convexity of the canonical norms of these spaces are of power type $\max(p,2)$ or $\max(p',2)$, respectively.
\item Quotients of uniformly convex spaces by closed subspaces are again uniformly convex.
\end{itemize}
\end{prop}

\begin{example}\label{ex:cpolar}While not apparent in the previous list, the uniform convexity of $Y$ and of $Y^\ast$ are independent of each other. As a simple example, consider $\R^2$ with the norm defined for $(x,y) \in \R^2$ by
\[\|(x,y)\|_C = \sup\left\{ \lambda > 0 \,\middle\vert\, (\lambda x, \lambda y) \in C\right\}\]
induced (all norms in $\R^d$ are of this form, see \cite[Thm.\ 15.2]{Roc70}) by the closed convex symmetric set
\[C:=\left\{ (x,y) \,\middle\vert\, \psi(x)+\psi(y) \ls 1\right\},\]
where $\psi$ is the Huber function of parameter $1/2$ defined by
\[
\begin{aligned}
\psi: \R &\to \R^+ \cup \{0\} \\
t &\mapsto 
\begin{cases} 
\frac{1}{2}|t|^2 & \text{if } |t| \ls \frac{1}{2} \\
\frac{1}{2}\left(|t| - \frac{1}{4}\right) & \text{if } |t| > \frac{1}{2}.
\end{cases}
\end{aligned}
\]
Now, the corresponding dual norm is induced \cite[Thm.\ 15.1]{Roc70} by the polar set $C^\circ$ of $C$ defined by
\[C^\circ=\left\{ (\bar x,\bar y) \,\middle\vert\, \bar x x + \bar y y \ls 1 \text{ for all }y\in C\right\},\]
so we can denote it by $\|\cdot\|_{C^\circ}$. In view of the definition of $C^\circ$, it is easy to convince oneself that $\|\cdot\|_{C^\circ}$ is uniformly convex; roughly, the influence of the rounded corners of $C$ will prevent the facets of $C^\circ$ from being completely flat, see Figure \ref{fig:cpolar}. However, $\|\cdot\|_C$ is clearly not uniformly convex. In fact, for norms in a Banach space the dual property to uniform convexity is uniform smoothness (in the sense that the limit defining the Fr\'{e}chet derivative exists uniformly in the point and direction taken) \cite[Prop.\ 5.1.18 and Cor.\  5.1.21]{BorVan10} and since $\psi_\eps \in \mathcal C^1$, $\|\cdot\|_C$ is uniformly smooth, which implies uniform convexity of $\|\cdot\|_{C^\circ}$.
\end{example}

\begin{figure}
\centering
 \includegraphics[width = 0.6\textwidth]{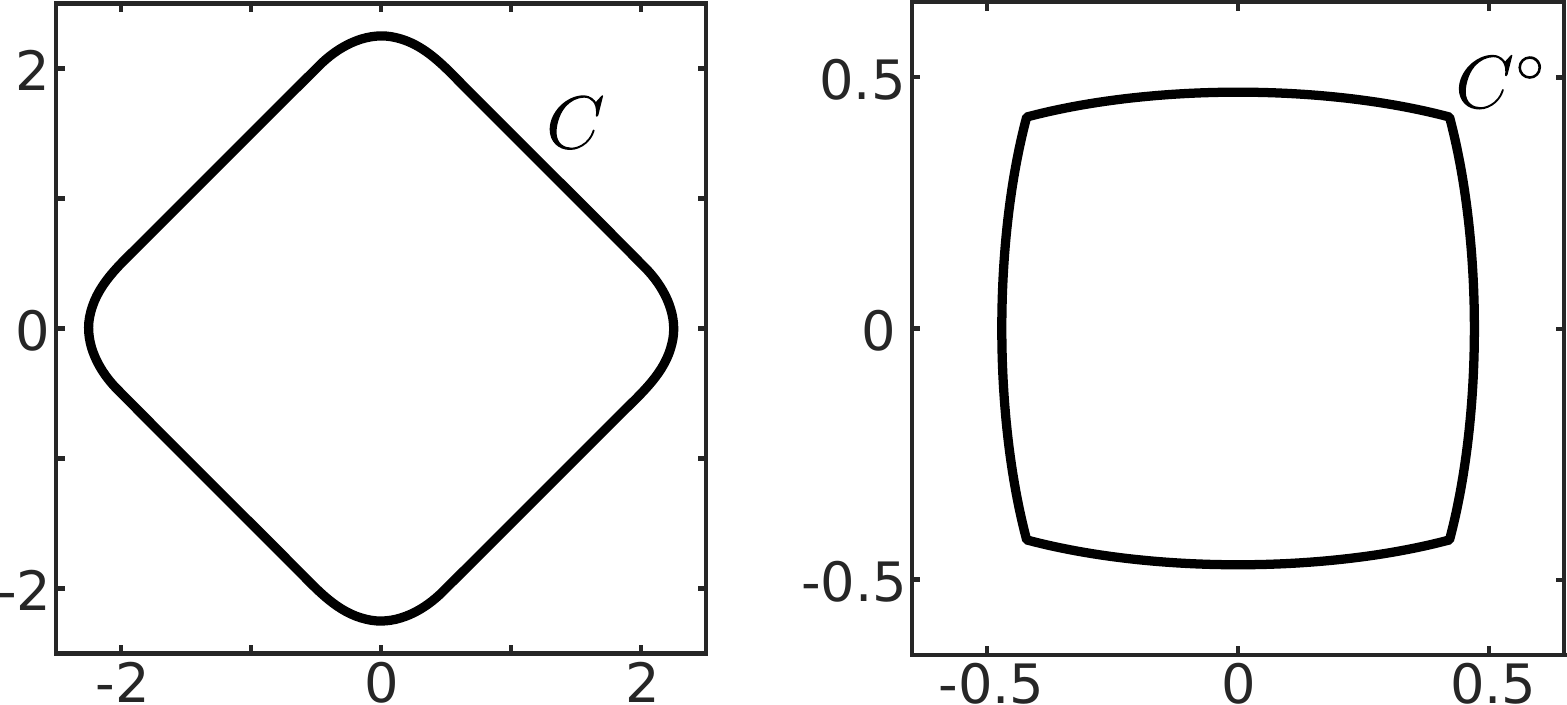}
 \caption{The unit ball $C$ of Example \ref{ex:cpolar} and its polar $C^\circ$, the unit ball of the dual space. The duality between uniform convexity and uniform smoothness also brings some intuition on uniform convexity of $Y^\ast$ being required for differentiability of $\|\cdot\|_Y^2$.}
 \label{fig:cpolar}
\end{figure}

Since we consider Fenchel duality for the minimization problem \eqref{eq:primalalphaw}, we will need the \emph{duality mapping} of $Y$, that is defined as
\begin{equation}\label{eq:dualitymap}
\begin{aligned}
j: Y &\to Y^\ast \\
g &\mapsto \partial \left(\frac{1}{2}\|\cdot\|^2_Y\right)(g)
\end{aligned}
\end{equation}
where, as before, $\partial$ denotes the subgradient. Note that $j$ is one-homogeneous. We make use of the following topological properties of $Y$ and its dual (for the proofs, see \cite[Ex\ 5.3.11, Thm.\ 5.4.6]{BorVan10}, \cite[Cor.\ 2.43]{SchKalHofKaz12}, \cite[Thm.\ 3.31]{Bre11} \cite[Prop.\ 32.22]{Zei90} and \cite[Thm.\ 5.3.7]{Meg98}).

\begin{prop}\label{prop:unifconv}
Let $Y$ be a Banach space. Then
\begin{itemize}
\item If $Y$ is uniformly convex, the function $\|\cdot \|_Y^{p}$ is uniformly convex on bounded sets for any $p > 1$. If additionally the modulus of uniform convexity of the norm of $Y$ is of power type $\tau$, then $\|\cdot \|_Y^p$ is globally uniformly convex for all $p\gs \tau$ .
\item Every uniformly convex Banach space is also reflexive, by the Milman-Pettis theorem.
\item If $Y^\ast$ is strictly convex, the duality mapping $j$ is single valued and the map $\frac{1}{2}\|\cdot\|^2_Y$ is G\^{a}teaux differentiable on $Y \setminus \{0\}$ with derivative $j$. If $Y^\ast$ is locally uniformly convex, then it is in fact Fr\'{e}chet differentiable. Moreover, if $Y$ is also locally uniformly convex, $j$ is invertible with inverse the duality mapping of $Y^\ast$ .
\item If $Y$ is locally uniformly convex, it has the Radon-Riesz property, that is if $y_n \rightharpoonup y$ is a weakly convergence sequence in $Y$ and if $\Vert y_n \Vert_Y \to \Vert y \Vert_Y$, then the convergence is strong.
\end{itemize}
\end{prop}

\paragraph{Perimeters and curvatures in a nonsmooth framework.}
In the rest of the article, we deal with convergence in the Hausdorff distance of the level-sets of minimizers of \eqref{eq:primalalphaw}. Let us define this mode of convergence:
\begin{definition}
 \label{def:hausdist}
 Let $E$ and $F$ two subsets of $\Omega$. The Hausdorff distance between $E$ and $F$ is defined as
 \begin{equation*}\begin{aligned}\label{eq:hausdist}
d_H(E, F)&=\max\set{\sup_{x \in E} d(x, F),\, \sup_{y \in F} d(y,E)} \\
&= \max\set{\sup_{x \in E}\, \inf_{y \in F} |x - y|,\, \sup_{y \in F}\, \inf_{x \in E} |x - y|}.\end{aligned}\end{equation*}
If $E_n$ is a sequence of subsets of $\Omega$, we say that $E_n$ Hausdorff converges to $F$ whenever $d_H(E_n, F) \to 0.$
\end{definition}
The minimizers of \eqref{eq:primalalphaw} belong to the space of functions of bounded variation, which has a strong relation with properties of their level-sets:
\begin{definition}
A function $u \in L^1_{\mathrm{loc}}(\R^d)$ is said to be of bounded variation (or belonging to $\BV(\R^d)$) if its distributional derivative is a Radon measure with finite mass, which we denote by $\TV{u}$. Equivalently, when 
\begin{equation}\label{eq:defTV}\TV{u} := \abs{Du}(\R^d) =  \sup \left\{\int_{\R^d} u\, \div z~ \dd x \, \middle\vert\, z \in \mathcal C^\infty_0(\R^d\, ;\,  \R^d), \|z\|_{L^\infty(\R^d)} \ls 1\right\} < + \infty.\end{equation}

We say that a set $E$ is of finite perimeter if its characteristic function $1_E$ is of bounded variation. In that case the perimeter is defined as 
\[\per(E):=\TV{1_E}.\]
Conversely, we can recover the total variation of a function $u \in \BV(\R^d)$ with compact support from the perimeter of its level-sets through the coarea formula \cite[Thm.\ 3.40]{AmbFusPal00} 
\begin{equation}\label{eq:coarea}
\TV{u} = \int_{-\infty}^{\infty} \per(\{u > s\}) \dd s = \int_{-\infty}^\infty \per(\{u < s\}) \dd s.
\end{equation}
\end{definition}

The main geometric tool used in the rest of the article is the isoperimetric inequality for sets of finite perimeter in $\R^d$ (see \cite[Thm.\ 14.1]{Mag12}, for example): 
\begin{prop}\label{prop:isopIneq}
Let $E \subset \R^d$ be a set of finite perimeter with $|E| < +\infty$. Then we have 
\begin{equation}\label{eq:isopIneq}
\frac{\per(E)}{|E|^{\frac{d-1}{d}}} \gs \Theta_d,\text{ where }\Theta_d:=\frac{\per(B(0,1))}{|B(0,1)|^{\frac{d-1}{d}}}=d |B(0,1)|^\frac{1}{d} = d^{\frac{d-1}{d}} \per(B(0,1))^{\frac{1}{d}},
\end{equation}
and equality holds if and only if $|E \Delta \Br|=0$ for some $x \in \R^d$ and $r>0$.
\end{prop}

We will also use extensively the notion of \emph{variational (mean) curvature}, defined as follows:
\begin{definition}
\label{def:varcurv}
 Let $E$ be a subset of $\R^d$ with finite perimeter. $E$ is said to have variational mean curvature $\kappa$ if $E$ minimizes the functional
 \[ F \mapsto \per(F) - \int_F \kappa. \]
\end{definition}
There is no uniqueness of the variational curvatures of a set. In fact, one can show that if $\kappa$ is a variational mean curvature for $E$, then for $f\gs 0$ in $E$ and $f \ls 0$ in $\R^d \setminus E,$ $\kappa +f$ is also a variational mean curvature for $E$. Nevertheless, in \cite{Bar94}, specific variational curvatures with particular desirable properties are introduced. Let us briefly sketch their construction:

\begin{prop}[{\cite[Thm.\ 2.1]{Bar94}}]
\label{prop:CurvL1}
Let $E$ be of finite perimeter in $\R^d$ and for $\lambda >0$, $h \in L^1(\R^d)$ with $h>0$ and $E_\lambda$ be a minimizer of 
\begin{equation}\label{eq:funclambda} F \mapsto \per(F) - \lambda \int_F h \end{equation}
among $F \subset E.$
Then, for $\lambda < \mu$, $E_\lambda \subset E_\mu$ up to a set of Lebesgue measure zero. That allows to define, for $x \in E$,
\begin{equation}\label{eq:kappavals} \kappa_E(x) := \inf \{\lambda h(x) >0 \  \vert \ x \in E_\lambda \}.\end{equation}
One can similarly define $\kappa_E$ outside $E$ by stating
\[\kappa_E(x) := -\kappa_{\R^d \setminus E}(x) \quad \text{ for } x\in \R^d \setminus E.\]
As built, $\kappa_E$ is a variational mean curvature for $E$. It minimizes the $L^1(\R^d)$ norm among variational curvatures, with $\|\kappa_E\|_{L^1(\R^d)}=2 \per(E)$ . 
\end{prop}

\begin{remark}The appearance of the density $h \in L^1(\R^d)$ is required for $\kappa_E$ to be well defined, since otherwise the functionals \eqref{eq:funclambda} would not be bounded below. Unfortunately, the curvatures obtained are not independent of $h$, even if their $L^1(\R^d)$-norm is optimal for each $h$. However, if $E$ is bounded we are allowed to choose $h(x)=1$ for all $x$ in $E$ or even in its convex envelope. The curvature obtained for such an $h$ minimizes all the $L^p(E)$ norms for $1<p<+\infty$, and its values on $E$ are uniquely defined by this minimizing property \cite[Thm.\ 3.2]{Bar94}. Consequently, there is a canonical choice for the variational curvature $\kappa_E$ inside $E$, and in the rest of the article we will use specific values of these variational curvatures only inside their corresponding sets.
\end{remark}

\begin{example}
\label{ex:curvball}
 Following Proposition \ref{prop:CurvL1} with $h(x) = 1_{B(0,R)}(x) + \frac{1}{|x|^{2d}} 1_{\R^d \setminus B(0,R)}(x) $, the ball $B(0,R)$ in $\R^d$ has a curvature
  \[ \kappa(x) = \frac{d}{R} 1_{B(0,R)}(x) -\frac{d-1}{|x|} 1_{\R^d \setminus B(0,R)}(x). \]
  We can show this by noting that for $\lambda > 0$, a minimizer of $F \mapsto \per(F) - \lambda |F|$ among $F \subset B(0,R)$ is $\emptyset$ for $\lambda \ls \frac dR$ and $B(0,R)$ for $\lambda \gs \frac dR.$
  Similarly, minimizers of 
  \[ F \mapsto \per(F) - \lambda \int_F h(x) \dd x\] among $F \supset B(0,R)$ are $B(0,r)$ with $r= \left( \frac{\lambda}{d-1} \right)^{\frac{1}{2d-1}}$, which taking into account $h$ as in \eqref{eq:kappavals} gives the second part of $\kappa.$ 
\end{example}

For further information about functions of bounded variation and sets of finite perimeter, see \cite{AmbFusPal00, Mag12}. An overview on variational curvatures and their interplay with the regularity of $\partial E$ can be found in \cite{GonMas94}.

\subsection{Organization of the paper}
We first present an example of noisy data for total variation denoising in the three-dimensional space in which the level-sets of the regularized solutions do not converge in Hausdorff distance to those of the noiseless data, regardless of the parameter choice used.

Motivated by this example, we study the existence and convergence of minimizers of the regularized problem \eqref{eq:primalalphaw} while keeping the dimension and integrability as general as possible. We compute then the dual problem and find that in the noiseless case its solutions strongly converge under the assumption of the standard source condition, and then study the effect of the noise by proving a quantitative stability estimate for these dual solutions.

Next, we see how the convergence of the dual solution and a parameter choice inequality arising from the stability estimate imply uniform weak regularity on the level-sets of the primal minimizers. Under the assumption of their compact support, this regularity makes equivalent the strong convergence of the primal minimizers in $L^1$ and the Hausdorff convergence of their level-sets.

We then explore whether this compact support can be derived from the problem itself. This turns out to be only possible for the exponent appearing in the Sobolev embedding of the space of bounded variation functions in the whole $d$-dimensional space.

Finally, we see how the previous analysis allows us to obtain analogous results in reasonable bounded domains, with Dirichlet or Neumann boundary conditions.

\section{The dimension matters: ROF denoising in 3D}\label{sec:rofCounterex}
We begin by justifying the need of generality in our formulation by showing through a counterexample that convergence of level-sets of minimizers of \eqref{eq:primalalphaw} does not necessarily hold when $A=\Id$, $q=\sigma=2$ and $d=3$. This corresponds to an straightforward extension to three dimensions of the Rudin-Osher-Fatemi (ROF) denoising model \cite{RudOshFat92}, a choice that has been made in some works, for example \cite{BehZhoEftAdi12}.

We recall that the level-set $\{u > s\}$ of value $s$ of the ROF solution $u$ for some data $f$ minimizes the functional
\begin{equation}\label{eq:roflevel}E \mapsto \alpha\per(E) - \int_E f-s,\end{equation}
which can be easily proved using the coarea formula \eqref{eq:coarea}.

The functions in our counterexample will be linear combinations of characteristic functions of two balls, so we begin by showing that in some situations the three-dimensional ROF problem can be solved explicitly for such data.
\begin{lemma}\label{lem:farness}
Assume that $f$ is of the form
\[f=c_1 1_{B(0,r_1)} + c_2 1_{B(x_0,r_2)},\]
with $c_1, c_2 >0$ as well as $r_1,r_2 >0$. Then there is a constant $D$ (depending on $r_1, r_2$) such that if $|x_0| > D$ the level-sets $E_s:=\{u >s\}$ of ROF denoising satisfy $E_s \subseteq B(0,r_1) \cup B(x_0,r_2)$ for each $s \gs 0$.
\end{lemma}
\begin{proof}
Without loss of generality, we may assume that $\alpha=1$, the other cases being obtained by rescaling of $f$ and $s$.

First, using the symmetry of revolution of the problem along the axis defined by the origin and $x_0$ and its strict convexity, we have that the unique solution of the ROF problem also possesses this symmetry, implying that each $E_s$ has the same symmetry.

Then we notice that because $f-s \in L^\infty$ we may apply regularity theorems for $\Lambda$-minimizers of the perimeter \cite[Thm.\ 26.3]{Mag12} to obtain that the boundaries $\partial E_s$ are in fact $C^{1,\alpha}$ surfaces for $\alpha < 1/2$.

On the other hand, since $E_s$ minimize \eqref{eq:roflevel} we have that $E_s$ must be contained (up to a set of measure zero) in $E_0$. Indeed, by minimality of each set, we have
\begin{equation}\label{eq:complevels}\begin{aligned}
\per(E_s)-\int_{E_s} f-s &\ls \per(E_s \cap E_0)-\int_{E_s \cap E_0} f-s, \text{ and }\\
\per(E_0)-\int_{E_0} f &\ls \per(E_s \cup E_0)-\int_{E_s \cup E_0} f.
\end{aligned}\end{equation}
Summing these inequalities, using the inequality (see \cite[Lem.\ 12.22]{Mag12})
\begin{equation}\label{eq:unionint}\per(E_s \cap E_0) + \per(E_s \cup E_0)\ls \per(E_s)+\per(E_0)\end{equation}
and linearity of the integrals, we end up with $s|E_s \setminus E_0| \ls 0$, so that $|E_s \setminus E_0| = 0$. Combining with this fact with the regularity, we only need to prove the claim for $E_0$. 

Moreover, since connected components of $E_0$ are also minimizers of \eqref{eq:roflevel}, we may also assume that $E_0$ is connected. We can distinguish three cases: $E_0$ could intersect neither $B(0,r_1)$ nor $B(x_0,r_2)$, one of them, or both.

The first case cannot happen, since if $E_0$ is nonempty, it must intersect either $B(0,r_1)$ or $B(x_0,r_2)$. To prove this claim, assume otherwise and notice that since $E_0$ minimizes \eqref{eq:roflevel}, it admits $f$ as a variational curvature. Since $f \gs 0$ and $f=0$ on $E_0$ by assumption, we would have that $E_0$ also admits the zero function as a variational curvature, making it an absolute minimizer of perimeter in $\R^3$, which can only be the empty set or the whole $\R^3$.

For the second case we have that if $E_0$ intersects one of the balls (assumed to be $B(0,r_1)$ without loss of generality) but not the other, then it must contain the whole $B(0,r_1)$. To prove this, we note that by the computation in Example \ref{ex:curvball}, $B(0,r_1)$ admits an optimal variational curvature such that
\[\kappa_{B(0,r_1)}1_{B(0,r_1)}=\frac{3}{r_1}1_{B(0,r_1)}=\frac{\per(B(0,r_1)}{|B(0,r_1)|}1_{B(0,r_1)}.\]
As before, we can use optimality to write
\begin{equation}\label{eq:compball}\begin{aligned}
\per(B(0,r_1))-\int_{B(0,r_1)} \kappa_{B(0,r_1)} &\ls \per(B(0,r_1) \cap E_0)-\int_{B(0,r_1) \cap E_0} \kappa_{B(0,r_1)}, \text{ and }\\
\per(E_0)-\int_{E_0} f &\ls \per(B(0,r_1) \cup E_0)-\int_{B(0,r_1) \cup E_0} f,
\end{aligned}\end{equation}
which leads to 
\[\left(c_1-\frac{3}{r_1}\right)|B(0,r_1) \setminus E_0| \ls 0,\]
so as long as $c_1 > 3/r_1$, we have that $B(0,r_1) \subseteq E_0$. We are left with the case $c_1 \ls 3/r_1$, for which we will need the isoperimetric inequality \eqref{eq:isopIneq} that can be written as
\begin{equation}\label{eq:isopball}\per(E_0) \gs \per(B(0,r_0)), \text{ with } r_0=\left(\frac{3}{4 \pi}|E_0|\right)^{1/3},\end{equation}
with equality only when $E_0$ is a ball of radius $r_0$. Now, if $|E_0|\gs|B(0,r_1)|$ (or equivalently $r_0 \gs r_1$) then we must have $B(0,r_1) \subseteq E_0$, since otherwise we would have
\[\per(E_0)-\int_{E_0}f > \per(B(0,r_0))-c_1|B(0,r_1)|=\per(B(0,r_0))-\int_{B(0,r_0)}f,\]
contradicting minimality of $E_0$ in \eqref{eq:roflevel}. If on the other hand $r_1 > r_0$, we obtain
\begin{equation}\begin{aligned}
\per(E_0)-\int_{E_0}f &= \per(E_0) - c_1|E_0 \cap B(0,r_1)| \\
                      &\gs \per(B(0,r_0)) - c_1|E_0|\\
                      &\gs \per(B(0,r_0)) - \frac{3}{r_1}|E_0|\\
                      &=4\pi r_0^2 - \frac{3}{r_1}\left(\frac{4}{3}\pi r_0^3\right)\\
                      &=4\pi r_0^2 \left(1-\frac{r_0}{r_1}\right)>0,
\end{aligned}\end{equation}
and this computation contradicts minimality of $E_0$, since it implies that it has strictly higher energy in \eqref{eq:roflevel} than the empty set.

Therefore, we end up with $B(0,r_1) \subseteq E_0$ but $E_0 \cap B(x_0, r_2) = \emptyset$. We must in fact have $E_0 = B(0,r_1)$, since otherwise the isoperimetric inequality \eqref{eq:isopball} would imply that $B(0,r_1)$ has a smaller perimeter than $E_0$ and, since $\restr{f}{E_0 \setminus B(0,r_1)}=0$, also strictly lower energy in \eqref{eq:roflevel}.

Finally we are left with the third case, in which $E_0$ is connected and intersects both balls. Using the symmetry and regularity, we have that $\partial E_0 \setminus \left( \partial B(0,r_1) \cup \partial B(x_0,r_2)\right)$ contains a minimal surface (of class $\mathcal C^{1,\alpha}$, as before) which is bounded by circles contained on planes orthogonal to $x_0$ and of radius less than or equal to $r_1$ and $r_2$ respectively. In fact, Schauder regularity theorems for elliptic equations can be used to obtain that this surface is $\mathcal C^\infty$ \cite[Thm.\ 27.3]{Mag12}. We can then conclude that this situation is impossible by applying classical results on necessary conditions for the existence of minimal surfaces bounded by planar curves (circles, in this case) \cite{Nit64a, Nit64b}. 
\end{proof}

\begin{remark}
The articles \cite{Nit64a, Nit64b} are likely the first in the direction of understanding from which distance $D$ any minimal surface spanning two orthogonal circles of radii $r_1$ and $r_2$ cannot be connected, providing $D\ls 3 \max(r_1, r_2)$, while the more recent \cite{Ros98} improves the bound to $D\ls 2 \max(r_1, r_2)$.
\end{remark}

\begin{remark}\label{rem:notonlyzero}
A closer examination of the arguments above shows that we have actually proved that each connected component of $E_0$ equals either $B(0,r_1)$ or $B(x_0,r_2)$. In fact, the arguments used for components that only intersect one ball also extend to components of $E_s$ with $s>0$ by just replacing $c_1$ by $c_1 - s$, so that in fact each connected component of $E_s$ equals either $B(0,r_1)$ or $B(x_0,r_2)$.
\end{remark}

\begin{remark}
In fact, Lemma \ref{lem:farness} can be proved without making use of the strong $\mathcal C^{1,\alpha}$ regularity. After developing the weak regularity tools that it requires, we will present in Section \ref{sec:infinity} a self-contained proof of this lemma, with the only price to pay being a worse control on $D$.
\end{remark}

\begin{example}\label{ex:3Dcounterex}
Assume that $\Omega \subset \R^3$ is bounded. In this situation, we consider denoising of the function $f=1_{B(0,1)}$ and a family of perturbations 
\[w_n:=c_n 1_{B(x_0, r_n)}, \text{ with }x_0=(3,0,0)\text{ and }r_n \ls 1.\]
Notice that $\|w_n\|_{L^2}=(\frac{4\pi}{3} c_n^2 r_n^3)^{1/2}$.
By Lemma \ref{lem:farness} and Remark \ref{rem:notonlyzero} we can compute the solution of \eqref{eq:primalalphaw} explicitly in this case, which will necessarily be of the form
\[u_n = b_n 1_{B(0,1)} + s_n 1_{B(x_0,r_n)},\]
and optimality provides 
\[b_n = \left( 1 - \frac{3}{2}\alpha_n \right)^+, \text{ and } s_n = \left(c_n - \frac{3}{2 r_n}\alpha_n \right)^+.\]
The goal is then to show that there is a choice of $c_n$ and $r_n$ such that $\|w_n\|_{L^2}$ goes to zero fast enough, but for which $s_n$ does not vanish, so the perturbation appears in the level sets of the denoised function.

In \cite{IglMerSch18} Hausdorff convergence of level-sets was proved under the condition $\|w_n\|_{L^2}/\alpha_n \ls C$. In the limit case $\alpha_n = C \|w_n\|_{L^2}$, then it suffices to choose $r_n=1/n$ and $c_n=n$, in which case we have $\|w_n\|_{L^2}=C/\sqrt{n}$, $\alpha_n=C/\sqrt{n}$, and $s_n=n-C\sqrt{n}$, as required.

One could think that by applying more aggressive regularization (a case still covered in the condition $\|w_n\|_{L^2}/\alpha_n \ls C$) convergence of level-sets could be restored. In fact, this is not the case. To see this, assume that we are given a strictly increasing function $f(t)\ls t$, with $f(0)=0$. Then we can choose
\[\alpha_n=\frac{2}{3 n},\ c_n = \frac{1}{n f\left(\frac{1}{n}\right)^2}+1,\text{ and }r_n=f\left(\frac{1}{n}\right)^2.\]
With this choice, we have $c_n \gs n+1$ and $s_n = 1$, preventing convergence of the level-sets corresponding to values less than one. Furthermore, if $n$ is large enough so that $1/n + f(1/n)\ls \sqrt{3/(4 \pi)}$ we also have
\[\|w_n\|_{L^2}=\sqrt{\frac{4 \pi}{3}} c_n r_n^{\frac{3}{2}}=\sqrt{\frac{4 \pi}{3}}\left(\frac{1}{n}f\left(\frac{1}{n}\right) + f\left(\frac{1}{n}\right)^2\right) \ls f\left(\frac{1}{n}\right)=f\left(\frac{3}{2}\alpha_n \right).\]
Since $f$ was arbitrary among sublinear functions, the resulting sequences $w_n$ `defeat' any sensible parameter choice rule based on the $L^2$ norm used in the data term. 

In the sequel we will see that convergence can be restored for domains of any dimension, if the error is measured in an adequate $L^q$ space with $q \neq 2$.
\end{example}

\section{Convergence of primal and dual solutions}\label{sec:primaldual}
We start by studying existence of minimizers for \eqref{eq:primalalphaw} and their convergence, the dual problem, and convergence of the corresponding dual solutions.

\begin{prop}\label{prop:linearexistence}Assume there is at least one solution $u_0$ of $Au=f$ with $\TV{u_0} < +\infty$, and that either $q=d/(d-1)$ or $\Omega$ is bounded. Then the problem \eqref{eq:primalalphaw} possesses at least one minimizer. If $A$ is injective, the minimizer is unique.\\
In addition, if $\alpha_n \to 0$ and $w_n \in Y$ are such that $\Vert w_n \Vert^\sigma_Y/\alpha_n$ is bounded, and if $u_n$ are minimizers of
\[\inf_{u \in L^q(\Omega)} \frac{1}{\sigma}\|Au-f-w_n\|^\sigma_{Y}+\alpha_n \TV{u}, \]
then we have (up to possibly taking a subsequence) the weak convergence $u_n \rightharpoonup u^\dag$ in $L^{d/(d-1)}(\Omega)$, where $u^\dag$ is a solution of $Au=f$ of minimal total variation among such solutions. Furthermore, if $\bar q < d/(d-1)$, we also have $u_n \to u^\dag$ in the (strong) $L^{\bar q}_{\mathrm{loc}}(\R^d)$ topology.
\end{prop}
\begin{proof}For the existence statement, let $(u_k)$ be a minimizing sequence. Since $u_k \in L^q(\R^d)$, we have that $u_k \in L^1_{\text{loc}}(\R^d)$, so the Sobolev inequality for $\BV$ functions \cite[Thm.\ 3.47]{AmbFusPal00} provides us with constants $c_k$ such that
\[\|u_k - c_k\|_{L^{\frac{d}{d-1}}(\R^d)} \ls C \TV{u_k},\]
and we must have $c_k = 0$ since $u_k \in L^q(\R^d)$. The $u_k$ being a minimizing sequence, $\TV{u_k}$ is bounded so using a standard compactness result in $\BV$ \cite[Thm.\ 3.23]{AmbFusPal00} and the Banach-Alaoglu theorem we obtain that $u_k$ converges (up to possibly taking a subsequence) weakly in $L^{d/(d-1)}(\R^d)$ and strongly in $L^1_{\text{loc}}(\R^d)$ to some limit $u \in L^{d/(d-1)}(\R^d)$. 

If $q=d/(d-1)$, since $A:L^q(\Omega) \to Y$ is bounded linear, $Au_k$ also converges weakly to $Au$ in $Y$. Lower semicontinuity of the norm with respect to weak convergence, and of the total variation with respect to strong $L^1_{\text{loc}}(\R^d)$ convergence \cite[Remark 3.5]{AmbFusPal00} imply that $u$ realizes the infimal value in \eqref{eq:primalalphaw}, and we obtain that $u$ is a solution of \eqref{eq:primalalphaw}. 

If on the contrary $q<d/(d-1)$, we cannot conclude that $u \in L^q(\Omega)$ unless $|\Omega| \ls +\infty$, in which case $\|u\|_{L^q(\Omega)}\ls |\Omega|^{1/q-(d-1)/d}\|u\|_{L^{d/(d-1)}(\Omega)} < +\infty$. This kind of inequality also provides boundedness of $u_n$ in $L^q(\Omega)$ and therefore the convergence of $Au_k$ to $Au$.

The proof of uniqueness, using injectivity of $A$ and strict convexity of the data term, follows entirely along the lines of the $L^2$ case treated in \cite[Prop.\ 1]{IglMerSch18}.

Existence of $u^\dag$ is covered in \cite[Thm.\ 3.25]{SchGraGroHalLen09}. Since $u_n \in L^q(\R^d)$ and $\Vert w_n \Vert^\sigma_Y/\alpha_n$ is bounded implies $\TV{u_n}$ is also bounded, we have that $\|u_n\|_{L^{d/(d-1)}}$ is again bounded \cite[Thm.\ 3.47]{AmbFusPal00}, giving weak convergence of a subsequence. The strong convergence statement relies on compact embeddings for $\BV$ along similar lines, and a proof can be found in \cite[Thm.\ 5.1]{AcaVog94}.
\end{proof}

\begin{remark}For the counterexample of Section \ref{sec:rofCounterex}, we have that $2=q > d/(d-1) = 3/2$. Existence of solutions can still be proven by the above straightforward methods, but only because $A$ is the identity, so that the data term provides a bound in $L^q$.
\end{remark}

\begin{prop}\label{prop:lineardual}The Fenchel dual problem of \eqref{eq:primalalphaw}, writes, for $\alpha > 0$, 
\begin{equation}\label{eq:dualalphaw}\sup_{\substack{p \in Y^\ast \\ A^\ast p \in \partial \TV{0}}} \scal{p}{f+w}_{(Y^\ast, Y)}-\frac{\alpha^{\frac{1}{\sigma-1}}}{\sigma'}\|p\|_{Y^\ast}^{\sigma'},\tag{$D_{\alpha,w}$}\end{equation}
where $1/\sigma+1/\sigma'=1$. Moreover, strong duality holds, the maximizer $p_{\alpha, w}$ of \eqref{eq:dualalphaw} is unique, and the following optimality condition holds:
\begin{equation}\label{eq:dualalphawdiff}v_{\alpha,w}:=A^\ast p_{\alpha, w} \in \partial \TV{u_{\alpha, w}}.\end{equation}

Here, the subgradient is understood to be with respect to the $\left(L^q(\Omega),L^{q'}(\Omega)\right)$ pairing, so that $\partial \TV{0} \subset L^{q'}(\Omega)$.
\end{prop}
\begin{proof}By the assumptions on $Y$, we have that the duality mapping $j$ defined in \eqref{eq:dualitymap} is single valued and invertible with inverse the duality mapping of $Y^\ast$, and that the map $\frac{1}{2}\|\cdot\|^2_Y$ is G\^{a}teaux differentiable with derivative $j$. Defining the functional 
\begin{equation*}
\begin{aligned}
G: Y &\to \R \\
g &\mapsto \frac{1}{\sigma \alpha}\|g-(f+w)\|^\sigma_Y,
\end{aligned}
\end{equation*}
its conjugate is
\[G^\ast(p)=\sup_{g \in Y} \scal{p}{g}_{(Y^\ast, Y)}-\frac{1}{\sigma \alpha}\|g-(f+w)\|^\sigma_Y,\]
and by the G\^{a}teaux differentiability we may take a directional derivative in direction $h \in Y$ to find that at a purported maximum point $g_0$,
\[\scal{p}{h}_{(Y^\ast, Y)} -\frac{1}{\alpha}\|g_0-(f+w)\|^{\sigma-2}_Y \scal{j\big(g_0-(f+w)\big)}{h} = 0,\]
or, since $h$ was arbitrary,
\[p=\frac{1}{\alpha}\|g_0-(f+w)\|^{\sigma-2}_Y j\big(g_0-(f+w)\big),\]
from which we get, computing norms on both sides and taking into account 
\[\|j(g_0 - (f+w))\|_{Y^\ast}=\|g_0 - (f+w)\|_{Y},\] that
\[p=\frac{1}{\alpha}\big(\alpha \|p\|_{Y^\ast}\big)^{\frac{\sigma-2}{\sigma-1}}j\big(g_0-(f+w)\big),\]
and inverting $j$ we end up with
\[g_0 = (f+w) + \alpha^{\frac{1}{\sigma-1}} \|p\|_{Y^\ast}^{\frac{2-\sigma}{\sigma-1}}j^{-1}(p).\]
Since $Y$ is assumed uniformly convex, the function to be maximized was strictly concave and differentiable and $g_0$ provides the only solution. With it we can compute, taking into account that $\|j^{-1}(p)\|_Y = \|p\|_{Y^\ast}$ and $\scal{p}{j^{-1}(p)}_{(Y^\ast, Y)}=\|p\|_{Y^\ast}^2$,
\[\scal{p}{g_0 - (f+w)}_{(Y^\ast, Y)}=\alpha^{\frac{1}{\sigma-1}} \|p\|_{Y^\ast}^{\frac{2-\sigma}{\sigma-1}+2}=\alpha^{\frac{1}{\sigma-1}} \|p\|_{Y^\ast}^{\sigma'},\]
\begin{equation*}\begin{aligned}\frac{1}{\sigma\alpha}\|g_0-(f+w)\|^\sigma_Y&=\frac{1}{\sigma \alpha}\left\|\alpha^{\frac{1}{\sigma-1}}\|p\|_{Y^\ast}^{\frac{2-\sigma}{\sigma-1}} j^{-1}(p)\right\|_{Y}^\sigma\\
&=\frac{1}{\sigma}\alpha^{\frac{\sigma}{\sigma-1}-1}\left|\|p\|_{Y^\ast}^{\frac{2-\sigma}{\sigma-1}+1}\right|^\sigma\\
&=\frac{1}{\sigma}\alpha^{\frac{1}{\sigma-1}}\|p\|_{Y^\ast}^{\sigma'},
\end{aligned}\end{equation*}
so that finally
\[G^\ast(p)=\scal{p}{f+w}_{(Y^\ast, Y)} + \alpha^{\frac{1}{\sigma-1}}\left(1-\frac{1}{\sigma}\right)\|p\|_{Y^\ast}^{\sigma'}.\]
The rest follows by Fenchel duality in a general pair of Banach spaces \cite[Thm.\ 4.4.3, p.\ 136]{BorZhu05} applied to the choices (in their notation) $X=L^q(\Omega)$ and $Y$, $f(\cdot)=\TV{\cdot}$, $g=G$ and $A$ as above. The functional $\mathrm{TV}^\ast$ is computed in \cite[Thm.\ 1]{IglMerSch18}, resulting in the indicator function of $\partial \TV{0}$. Uniqueness holds because by the assumptions on $Y^\ast$, we have that $\|\cdot\|_{Y^\ast}^{\sigma'}$ is strictly convex.
\end{proof}

Following the scheme laid out in \cite{ChaDuvPeyPoo17, IglMerSch18} for the convergence of level-sets, we now prove strong convergence of the dual maximizers corresponding to noiseless data. It relies on the following source condition:
\begin{equation}\label{eq:sourceCond}
\mathcal{R}(A^\ast) \cap \partial \TV{u^\dag} \neq \emptyset,
\end{equation}
where $\mathcal R(A^\ast)$ denotes the range of the adjoint operator $A^\ast.$
\begin{prop}\label{prop:dualconv}
Assume that the source condition \eqref{eq:sourceCond} holds. Then there is a unique maximizer $p_{0,0}$ of the problem
\[ \sup_{A^\ast p \in \partial \TV{0}} \scal{p}{f}_{(Y^\ast, Y)}\] and with minimal $Y^\ast$ norm. Furthermore, in the absence of noise ($w=0$) the sequence $p_{\alpha,0}$ of maximizers of the dual problem \eqref{eq:dualalphaw} converges strongly in $Y^\ast$ to it.
\end{prop}
\begin{proof}
The existence of $p_{0,0}$ follows along the same steps as the Hilbert space case treated in \cite[Lem.\ 2]{IglMerSch18}, while uniqueness is a consequence of the strict convexity of $Y^\ast$ (and therefore of powers of its norm).

By optimality in their corresponding maximization problems, we have 
\begin{equation}\label{eq:comp1}\scal{p_{\alpha, 0}}{f}_{(Y^\ast, Y)}-\frac{\alpha^{\frac{1}{\sigma-1}}}{\sigma'}\|p_{\alpha,0}\|_{Y^\ast}^{\sigma'} \gs \scal{p_{0, 0}}{f}_{(Y^\ast, Y)}-\frac{\alpha^{\frac{1}{\sigma-1}}}{\sigma'}\|p_{0,0}\|_{Y^\ast}^{\sigma'},\end{equation}
and
\begin{equation}\label{eq:comp2}\scal{p_{0, 0}}{f}_{(Y^\ast, Y)} \gs \scal{p_{\alpha, 0}}{f}_{(Y^\ast, Y)}.\end{equation}
Summing these inequalities we obtain $\|p_{\alpha,0}\|_{Y^\ast} \ls \|p_{0,0}\|_{Y^\ast}$. Since $Y^\ast$ is uniformly convex, it is also reflexive and the sequence $p_{\alpha,0}$ can be assumed \cite[Cor.\ 3.30]{Bre11} (up to taking a subsequence) to converge weakly in $Y^\ast$ to some limit $p^\ast$. Furthermore $A^\ast p^\ast \in \partial \TV{0}$ by weak closedness of subgradients in Banach spaces \cite[Cor.\ I.5.1, p.\ 21]{EkeTem99}. Passing to the limit in both inequalities we obtain
\[\scal{p^\ast}{f}_{(Y^\ast, Y)} = \scal{p_{0, 0}}{f}_{(Y^\ast, Y)},\]
so that $p^\ast$ is a maximizer of $p \mapsto \scal{p}{f}_{(Y^\ast, Y)}$ over $p$ such that $A^\ast p \in \partial \TV{0}$. Using \eqref{eq:comp1} and weak lower semicontinuity of the norm we get that
\begin{equation}\label{eq:lscnorm}\|p^\ast\|_{Y^\ast} \ls \lim\inf \|p_{\alpha, 0}\|_{Y^\ast} \ls \|p_{0, 0}\|_{Y^\ast}.\end{equation}
This implies that $p^\ast$ is of $Y^\ast$ minimal norm, and since $\|\cdot\|_{Y^\ast}^{\sigma'}$ is strictly convex, such a minimizer is unique and we must have $p^\ast=p_{0,0}$ and the whole sequence $p_{\alpha,0}$ converging to it. Moreover, since $Y^\ast$ has the Radon-Riesz property, \eqref{eq:lscnorm} implies that the convergence is in fact strong in $Y^\ast$. 
\end{proof}

In the sequel, we will need stability estimates for solutions of the dual problem \eqref{eq:dualalphaw}, so that $p_{\alpha,w}$ can be related to $p_{\alpha,0}$, which was just proved to converge strongly. In the simple case where $\sigma=2$ and $Y$ is a Hilbert space $H$, the maximization to be performed corresponds to 
\begin{equation}\sup_{\substack{p \in H \\ A^\ast p \in \partial \TV{0}}} 2\scal{p}{\frac{f+w}{\alpha}}_{H}-\|p\|_{H}^{2},\end{equation}
which after adding the constant term $-\|(f+w) / \alpha\|^2_H$ has the same maximizers as the problem
\begin{equation}\sup_{\substack{p \in H \\ A^\ast p \in \partial \TV{0}}} -\left\|\frac{f+w}{\alpha}\right\|^2_H + 2\scal{p}{\frac{f+w}{\alpha}}_{H}-\|p\|_{H}^{2}=- \inf_{\substack{p \in H \\ A^\ast p \in \partial \TV{0}}} \left\|p - \frac{f+w}{\alpha}\right\|^2_H,\end{equation}
which is solved by computing the projection of $(f+w)/\alpha$ onto the convex set 
\[\left\{p \in H \,\middle|\, A^\ast p \in \partial \TV{0}\right\}.\]
Convexity of the set implies that this projection is nonexpansive, providing a straightforward stability estimate for this case.

In analogy with the Hilbert framework, we can define the functional
\begin{equation}\label{eq:V}V(p,g):=\frac{1}{\sigma'}\|p\|^{\sigma'}_{Y^\ast}- \scal{p}{g}_{(Y^\ast, Y)} + \frac{1}{\sigma}\|g\|_Y^{\sigma},\end{equation}
which in the case $\sigma=2$ is used in \cite{Alb96} to define a generalized projection for Banach spaces, mapping the dual space $Y^\ast$ onto $Y$. In the following we use the methods introduced in \cite{Alb96,AlbNot95} to derive the estimates we require.

\begin{prop}\label{prop:shiftingV}
For $g \in Y$ and any weak-* closed and convex set $K \subset Y^\ast$ The problem
\begin{equation}\inf_{p \in K} V(p,g)\end{equation}
has a unique solution, which we denote by $\pi_K(g)$. Furthermore, it satisfies
\begin{equation}\label{eq:projident}\scal{\pi_K(g)-q}{g-\|\pi_K(g)\|^{\sigma'-2}_{Y^\ast}\,j^{-1}(\pi_K(g))}_{(Y^\ast, Y)}\gs 0 \text{ for each }q \in K.\end{equation}
\end{prop}
\begin{proof}Existence follows by the Banach-Alaoglu theorem and closedness, while uniqueness is a consequence of the strict convexity of the function $\|\cdot\|^{\sigma'}_{Y^\ast}$.

For the second part, we have that 
\begin{equation}V(\pi_K(g),g)=\min_{p \in K}V(p,g),\end{equation}
and since we have G\^{a}teaux differentiability of the squared dual norm $\Vert \cdot \Vert^{2}_{Y^\ast}$ and that the duality mapping of $Y ^\ast$ is $j^{-1}$ by Proposition \ref{prop:unifconv}, we can differentiate $V$ at $(\pi_K(g),g)$ in its first argument in direction $q-\pi_K(g) \in Y^\ast$, to obtain
\begin{equation}\scal{q-\pi_K(g)}{\|\pi_K(g)\|^{\sigma'-2}_{Y^\ast}\,j^{-1}(\pi_K(g))}_{(Y^\ast, Y)}-\scal{q-\pi_K(g)}{g}_{(Y^\ast, Y)}\gs 0,\end{equation}
from which \eqref{eq:projident} follows directly.
\end{proof}

Since we have assumed that $Y^\ast$ is uniformly convex with modulus of uniform convexity of power type $\sigma'$, we have by Proposition \ref{prop:unifconv} that $\|\cdot\|^{\sigma'}_{Y^\ast}$ is also globally uniformly convex. This allows us to formulate stability estimates for the generalized projection:

\begin{prop}\label{prop:genproj}
We have the estimate:
\begin{equation}\label{eq:genprojest}
\|\pi_K(g_1)-\pi_K(g_2)\|_{Y^\ast} \ls \rho_{Y,\sigma}\left(\frac{1}{2}\|g_1 - g_2\|_Y\right),
\end{equation}
where $\rho_{Y,\sigma}$ is defined as the inverse of the function
\begin{equation}t \mapsto \frac{\delta_{\|\cdot\|^{\sigma'}_{Y^\ast}/\sigma'}(t)}{t},\end{equation}
 where $\delta_{\|\cdot\|^{\sigma'}_{Y^\ast}/\sigma'}$ is the modulus of uniform convexity of the functional $\|\cdot\|^{\sigma'}_{Y^\ast}/\sigma'$. In consequence, the solutions of \eqref{eq:dualalphaw} satisfy
 \begin{equation}\label{eq:parameterest}
\|p_{\alpha, w}-p_{\alpha, 0}\|_{Y^\ast} \ls \rho_{Y,\sigma}\left(\frac{\|w\|_Y}{2\alpha^{\frac{1}{\sigma-1}}}\right),
\end{equation}
\end{prop}
\begin{proof}
We denote $\phi(p)=\|p\|^{\sigma'}_{Y^\ast}/\sigma'$, so that $\phi$ is G\^{a}teaux differentiable with derivative \[d \phi(p)=\|\pi_K(p)\|^{\sigma'-2}_{Y^\ast}\,j^{-1}(\pi_K(p)).\] We compute
\begin{equation}
\begin{aligned}
&\scal{\pi_K(g_1)-\pi_K(g_2)}{d\phi(\pi_K(g_1))- d\phi(\pi_K(g_2))}_{(Y^\ast, Y)}\\
&\quad=\scal{\pi_K(g_1)-\pi_K(g_2)}{d\phi(\pi_K(g_1))-g_1}_{(Y^\ast, Y)} \\
&\qquad-\scal{\pi_K(g_1)-\pi_K(g_2)}{d\phi(\pi_K(g_2))-g_2}_{(Y^\ast, Y)} \\
&\qquad+ \scal{\pi_K(g_1)-\pi_K(g_2)}{g_1 - g_2}_{(Y^\ast, Y)}\\
&\quad\ls \scal{\pi_K(g_1)-\pi_K(g_2)}{g_1 - g_2}_{(Y^\ast, Y)} \\ 
&\quad\ls \|\pi_K(g_1) - \pi_K(g_2)\|_{Y^\ast} \|g_1 - g_2\|_Y,
\end{aligned}
\end{equation}
where we have used Proposition \ref{prop:shiftingV} twice and the Cauchy-Schwarz inequality. On the other hand, Lemma \ref{lem:unifmonoton} provides us with 
\begin{equation}
\scal{\pi_K(g_1)-\pi_K(g_2)}{d\phi(\pi_K(g_1))- d\phi(\pi_K(g_2))}_{(Y^\ast, Y)}\gs 2 \delta_{\phi}\left(\|\pi_K(g_1)-\pi_K(g_2)\|_{Y^\ast}\right),
\end{equation}
which combined with the above delivers \eqref{eq:genprojest}. Note that the inverse function $\rho_{Y,\sigma}$ is well defined, since the property $\delta_\phi(ct)>c^2\delta_\phi(t)$ for all $c>1$ \cite[Fact 5.3.16]{BorVan10} implies that $t \mapsto\delta_{\phi}(t)/t$ is strictly increasing.

Now, we notice that we can divide by $\alpha^{1/(\sigma-1)}$ in the problem \eqref{eq:dualalphaw}, to obtain the equivalent problem
\[\sup_{\substack{p \in Y^\ast \\ A^\ast p \in \partial \TV{0}}} \scal{p}{\frac{f+w}{\alpha^{\frac{1}{\sigma-1}}}}_{(Y^\ast, Y)}-\frac{1}{\sigma'}\|p\|_{Y^\ast}^{\sigma'},\]
which in turn has the same solutions as 
\[\inf_{\substack{p \in Y^\ast \\ p \in \partial \TV{0}}} V(p, \alpha^{-\frac{1}{\sigma-1}} (f+w)).\]
Using \eqref{eq:genprojest} with $g_1-g_2=\alpha^{-1/(\sigma-1)}w$, we get the expected estimate \eqref{eq:parameterest}.
\end{proof}

\begin{remark}\label{rem:hilbert}A straightforward computation shows that in the case $\sigma'=2$ and $Y=H$ a Hilbert space, we have for any $u,v \in H$
\[\left\|\frac{1}{2}(u+v)\right\|^2_H=\frac{1}{2}(\|u\|^2_H + \|v\|^2_H)-\frac{1}{4}\|u-v\|^2_H,\]
so that the best modulus of convexity of $\|\cdot\|^{2}_{H}/2$ is the function defined by $\delta_{\|\cdot\|^{2}_{H}/2}(t)=t^2/2$ and $\rho_{H,2}(t/2)=t$, recovering that the projection is nonexpansive, as used in \cite{IglMerSch18}.
\end{remark}

\section{Convergence of level-sets with assumed compact support}\label{sec:lsConv}
Our next goal is to relate the convergence of the sequence $p_{\alpha,w}$ with that of the level-sets. For the sake of clarity we assume throughout the section that the minimizers considered have a common compact support, and the possibility to lift this assumption will be discussed in Section \ref{sec:infinity}. We start by recalling some known properties of the subgradient of the total variation, which allow us to interpret the optimality condition \eqref{eq:dualalphawdiff} in terms of the level-sets of $u_{\alpha,w}$.
\begin{prop}
\label{prop:subgradTV}
 Let $u \in L^q(\R^d)$. Then, the following assertions are equivalent
 \begin{enumerate}
  \item\label{i:1} $v \in \partial \TV{u}$,
  \item\label{i:2} $v \in \partial \TV{0}$ and \[\int u v = \TV{u}\]
  \item\label{i:3} $v \in \partial \TV{0}$ and for a.e. $s$,
  \[\per(\{u>s\}) =\operatorname{sign}(s) \int_{\{u >s\}} v.\]
  \item\label{i:4} Almost every level-set $\{u>s\}$ minimizes
  \[E \mapsto \per(E) - \operatorname{sign}(s) \int_E v.\]
 \end{enumerate} 
\end{prop}
\begin{proof}The equivalence between statements \ref{i:1} and \ref{i:2} follows from the $(L^q, L^{q'})$ pairing used and the fact that $\TV{\cdot}$ is one-homogeneous, and a proof can be found in \cite[Lem.\ 10]{IglMerSch18}, for example. The equivalence between statements \ref{i:3}, \ref{i:4} and \ref{i:1} is a consequence of statement \ref{i:2} and the coarea formula, for a proof see \cite[Prop.\ 3]{ChaDuvPeyPoo17}.
\end{proof}

The proof of Hausdorff convergence of level-sets goes along the lines of the proof of Theorem 2 in \cite{IglMerSch18}, and is centered around uniform density estimates for the level-sets, that is, bounds on volume fractions of the type \[\frac{|\{u_{\alpha,w} >s\} \cap \Br|}{|\Br|}\gs C, \text{ for }x \in \partial\{u_{\alpha,w} >s\}\text{ and }r \text{ small},\] 
the uniformity referring to the fact that the constant in the right hand side should be independent of $\alpha$ and $w$, as long as they are related by a suitable parameter choice.

The first ingredient for such density estimates is the following comparison formula for intersections with balls, whose proof can be found, for example, in \cite[Lemma 3]{IglMerSch18}.  Remembering that $v_{\alpha,w} = A^\ast p_{\alpha,w}$, this formula applies to the level sets $\{u_{\alpha,w} >s\}$ by the last item of Proposition \ref{prop:subgradTV}.

\begin{lemma}\label{lem:compareB}
Let $E$ minimize the functional $F \mapsto \per(F) - \int_F v_{\alpha,w}$. Then for any $x$ and almost every $r$ we have 
\begin{equation} \per(E \cap \Br) - \int_{E\cap \Br} v_{\alpha,w} \ls 2 \per(\Br\, ;\, E^{(1)}).\label{eq:compareGen} \end{equation}
\end{lemma}

\begin{remark}Lemma \ref{lem:compareB} only depends on basic properties of the perimeter and minimality, so it's also valid when considering the relative perimeter $\per(F;\Omega)$ corresponding to Neumann boundary conditions (see Section \ref{sec:bounded}).
\end{remark}

With the comparison formula above, to arrive at density estimates one needs precise control on the term $\int_{E\cap \Br} v_{\alpha,w}$ as $r \to 0$. Since $v_{\alpha,w} = A^\ast p_{\alpha,w}$, this control is attained by combining the estimates of Proposition \ref{prop:genproj}, the equiintegrability of $v_{\alpha,0}$ and a parameter choice satisfying
\begin{equation}\label{eq:paramChoice}
 \frac{\|w\|_Y}{\alpha^{\frac{1}{\sigma-1}}} \ls 2 \frac{\|A^\ast\|}{\eta} \delta_{\|\cdot\|^{\sigma'}_{Y^\ast}/\sigma'}\left(\frac{\eta}{\|A^\ast\|}\right),\text{ with }\eta < \Theta_d,\
\end{equation}
$\Theta_d$ being the isoperimetric constant of Proposition \ref{prop:isopIneq}. As in Remark \ref{rem:hilbert}, in the case of $\sigma=2$, $d=2$ and $Y$ a Hilbert space $H$, the expression \eqref{eq:paramChoice} simplifies to $\|w\|_H\|A^\ast\|/\alpha \ls \eta < \Theta_2$, the parameter choice used in \cite{IglMerSch18}.

\begin{remark}Although the right hand side of the inequality \eqref{eq:paramChoice} might look involved, it just provides the optimal constant for the ratio $\|w\|_Y^{\sigma-1}/\alpha$ for which the convergence of level-sets can be proved by the methods presented. In particular, any choice such that $\|w\|_Y^{\sigma-1}/\alpha \to 0$ satisfies \eqref{eq:paramChoice}. The choice $\alpha \sim \|w\|_Y^{\sigma-1}$ also appears as a sufficient condition for linear convergence rates in Bregman distance when the source condition \eqref{eq:sourceCond} is assumed (see \cite[Thm.\ 3.42]{SchGraGroHalLen09} or \cite[Prop.\ 4.19]{SchKalHofKaz12}). One might wonder whether using an a posteriori choice rule is possible. Such linear convergence rates can also be proved using the Morozov discrepancy principle and under source conditions (\cite[Thm.\ 4.2]{Bon09}, \cite[Thm.\ 5.3]{AnzRam10}), but typically only $\|w\|_Y^{\sigma}/\alpha \to 0$ can be ensured for the ensuing parameters \cite[Thm.\ 4.5]{AnzRam10}, which is not enough to conclude \eqref{eq:paramChoice}.\end{remark}

Assuming that the parameter choice satisfies \eqref{eq:paramChoice}, we are now ready to prove the anticipated uniform density estimates:

\begin{theorem}\label{thm:densEsts}
Assume that the parameter choice satisfies \eqref{eq:paramChoice} and that the source condition \eqref{eq:sourceCond} holds. Let $E$ be a minimizer of
\[ F \mapsto \per(F) - \int_F v_{\alpha,w}. \]
Then, there exists $C >0$ and $r_0>0$, independent of $\alpha$ and $w$ such that for every ball $B(x,r_0)$ with $x\in \partial E$, one has, for any $r\ls r_0$,
\begin{equation} 
\frac{|E \cap B(x,r)|}{|B(x,r)|} \gs C \quad \text{and} \quad \frac{|E \setminus B(x,r)|}{|B(x,r)|} \gs C.
 \label{eq:densEsts}
\end{equation}

\end{theorem}
\begin{proof}
Using H\"{o}lder's inequality, that $q'= q/(q-1)\gs d$, the parameter choice \eqref{eq:paramChoice} and the estimate \eqref{eq:genprojest}, we obtain that for any $F \subset \R^d$ with $|F| < \infty$,
\begin{equation}\label{eq:distToNoiseless}\|v_{\alpha,w}-v_{\alpha,0}\|_{L^d(F)} \ls |F|^{\frac{q'-d}{q'd}}\|v_{\alpha,w}-v_{\alpha,0}\|_{L^{q'}(\R^d)} \ls |F|^{\frac{q'-d}{q'd}} \eta.\end{equation}
With this, we obtain
\begin{equation}\label{eq:integralBound}
\begin{aligned}
\left|\int_{E \cap \Br} v_{\alpha, w} \right| &\ls |E \cap \Br|^{\frac{d-1}{d}}\|v_{\alpha,w}\|_{L^d(E \cap \Br)}\\
& \ls |E \cap \Br|^{\frac{d-1}{d}} \left(\|v_{\alpha,0}\|_{L^d(E \cap \Br)}+\|v_{\alpha,w}-v_{\alpha,0}\|_{L^d(E \cap \Br)}\right)\\
&\ls |E \cap \Br|^{\frac{d-1}{d}} \left(\|v_{\alpha,0}\|_{L^d(E \cap \Br)}+|E \cap \Br|^{\frac{q'-d}{q'd}} \eta\right).
\end{aligned}
\end{equation}
Now, by Proposition \ref{prop:dualconv}, $v_{\alpha,0}$ converges strongly in $L^d$ as $\alpha \to 0$, and $|v_{\alpha,0}|^d$ is therefore equiintegrable. This implies that for each $\eps > 0$, there exists $r_\eps$ such that for all $r < r_\eps$ we have $\|v_{\alpha,0}\|_{L^d(E \cap \Br)} < \eps$. Moreover, by possibly reducing $r_\eps$ we may assume that 
\[|E \cap \Br|^{\frac{q'-d}{q'd}} \ls 1.\] 
Assuming $\eps < \Theta_d - \eta$ we can use then \eqref{eq:integralBound} in \eqref{eq:compareGen} and the isoperimetric inequality \eqref{eq:isopIneq} to obtain
\begin{equation}\label{eq:diffIneq1}
\begin{aligned}
2 \per(\Br\, ;\, E^{(1)}) &\gs \per(E \cap \Br) - |E \cap \Br|^{\frac{d-1}{d}} \left(\eps+\eta\right)\\
&\gs |E \cap \Br|^{\frac{d-1}{d}} \left(\Theta_d - \eps - \eta\right).
\end{aligned}
\end{equation}
Additionally, we have that for almost every $r$
\begin{equation}\label{eq:perHaus}
\per(\Br\, ;\, E^{(1)}) = \mathcal{H}^{d-1}(\partial\Br \cap E^{(1)}),
\end{equation}
which in turn is the derivative with respect to $r$ of the function $g(r):=|E \cap \Br|$, turning \eqref{eq:diffIneq1} into the variational inequality
\begin{equation}\label{eq:diffIneq2}
2 g'(r) \gs \left(\Theta_d - \eps - \eta\right) g(r)^{\frac{d-1}{d}}.
\end{equation}
Integrating on both sides taking into account $g(0)=0$, we end up with
\begin{equation}\label{eq:diffIneq3}
2 d g^{\frac{1}{d}}(r) \gs \left(\Theta_d - \eps - \eta\right)r
\end{equation}
which in turn implies the density estimate
\begin{equation}\label{eq:denstEstPf}
\frac{|E \cap \Br|}{|\Br|} \gs \frac{\left(\Theta_d - \eps - \eta\right)^d r^d}{(2d)^d|B(x,r)|}=\frac{\left(\Theta_d - \eps - \eta\right)^d}{(2d)^d|B(0,1)|},
\end{equation}
where the right hand side is uniform in $\alpha$, $w$, $r$ small enough and also in $x$.
\end{proof}

\begin{remark}
Note that if $q < d/(d-1)$ (that is, $q'> d$), the second term inside the parenthesis in the right hand side of \eqref{eq:integralBound} tends to zero as $r \to 0$, which implies that in this case the density estimates still hold for any parameter choice (see \eqref{eq:paramChoice}) that ensures that $\|w\|_Y / \alpha^{1/(\sigma-1)}$ remains finite.
\end{remark}

Combining the compact support assumption with the density estimates of Theorem \ref{thm:densEsts}, we arrive at the desired convergence result.

\begin{theorem}\label{thm:hausConv}
Let $f$ and $A$ satisfy \eqref{eq:sourceCond}, $\alpha_n, w_n \to 0$ satisfying \eqref{eq:paramChoice} and $ u_n := u_{\alpha_n,w_n}$ the corresponding minimizer of \eqref{eq:primalalphaw}. We assume that all the $u_n$ have a common compact support (we will see in Section \ref{sec:infinity} how to lift this artificial assumption). Then, for almost every $s \in \R$, as $n$ grows to infinity, the level-sets $\{u_n>s\}$ converge to $\{u^\dag>s\}$ in the sense of Hausdorff convergence.
\end{theorem}
\begin{proof}
We saw in Proposition \ref{prop:linearexistence} that $u_n \to u$ in $L^1_{\mathrm{loc}}.$ Combined with the compact support assumption for $u_n$, it leads to the full $L^1$ convergence. This implies, using Fubini's theorem (see \cite[Section 3.1]{IglMerSch18}) that for almost every $s$,
\[ \left \vert \{u_n >s\} \Delta \{u^\dag > s \} \right \vert \to 0.\]
Now, let us assume that the Hausdorff distance between these two level-sets does not go to zero. That means, using the definition of this distance, that there exists a constant $L>0$ and either a sequence of points $x_n \in \{u_n > s\}$ such that
$d(x_n,\{u^\dag >s\}) > L$ or a sequence $y_n \in \{u^\dag > s\}$ such that $d(y_n,\{u_n > s\}) >L$. We will treat the first case. One can assume that $x_n \in \partial \{u_n>s\}$. 

Using then the density estimates \eqref{eq:densEsts}, one concludes that for $r \ls \min(r_0,L)$,
\[\left \vert B(x_n,r) \cap \{u_n>s\} \right \vert \gs C |B(x_n,r)|. \]
On the other hand, since $r \ls L$, one has $ B(x_n,r) \cap \{u^\dag >s\} = \emptyset$ which implies that $B(x_n,r) \cap \{u_n>s\} \subset \{u_n >s\} \Delta \{u^\dag > s \}$ and contradicts the $L^1$ convergence.

The second case is treated similarly, but the contradiction is obtained using the density estimates on $\{u^\dag >s\}$.
\end{proof}

\section{Behavior at infinity}\label{sec:infinity}

We now discuss whether it is possible to remove the assumptions on compact support of the solutions that were used in the previous section. In view of the proof of Theorem \ref{thm:hausConv}, this amounts to being able to infer that $u_{\alpha,w} \to u^\dag$ strongly in $L^1(\Omega)$.

\subsection{The critical case}

The following lemma, analogous to \cite[Lemma 5]{IglMerSch18}, tells us that this is indeed possible for the critical exponent $q=d/(d-1)$, with the same parameter choice as in Section \ref{sec:lsConv}.
\begin{lemma}\label{lem:compsupp} Let $q=d/(d-1)$, and assume \eqref{eq:paramChoice} and \eqref{eq:sourceCond}. Then, the elements of
\begin{equation}\label{eq:curvls2-2}
\mathcal{E}:=\set{E \subset \Omega \, \middle | \, \per(E)=\int_E v_{\alpha, w} },
\end{equation}
have the following properties:
\begin{enumerate}
 \item There exists a constant $C > 0$ such that for all $E \in \mathcal{E}$, $\per(E) \leq C$,
 \item There exists a constant $R > 0$ such that for all $E \in \mathcal{E}$, $E \subseteq \mathcal{B}(0,R)$.
\end{enumerate}
\end{lemma}
\begin{proof}
The proof is very similar to what is done in \cite{ChaDuvPeyPoo17, IglMerSch18}. 

Here, by Proposition \ref{prop:dualconv}, we have that $v_{\alpha,0} \to v_0$ strongly in $L^{q'}=L^{d}$, and therefore the family $(v_{\alpha,0})$ is $L^{d}(\R^d)$-equiintegrable, which in particular means that for every $\eps > 0$, one can find a ball $B(0,R)$ such that 
\[\int_{\R^d \setminus B(0,R)} |v_{\alpha,0}|^{d} \ls \eps.\]
Then, for every $E$ with finite mass belonging to $\mathcal E$ and provided $\alpha$ and $w$ satisfy \eqref{eq:paramChoice} we get as in \eqref{eq:distToNoiseless} that
\begin{align*}
 \per(E) &\ls \left \vert \int_E (v_{\alpha,w} - v_{\alpha,0}) \right \vert + \left \vert \int_{E \cap B(0,R)} v_{\alpha,0} \right \vert + \left \vert \int_{E \setminus B(0,R)} v_{\alpha,0} \right \vert \\
 & \ls \eta |E|^{\frac{d-1}{d}} + |B(0,R)|^{\frac{d-1}{d}} \Vert v_{\alpha,0} \Vert_{L^{d}} + |E \setminus B(0,R)|^{\frac{d-1}{d}}  \eps \\
 & \ls \left( \eta + \sup_{\alpha} \Vert v_{\alpha,0} \Vert_{L^{d}} \right) |B(0,R)|^{\frac{d-1}{d}}  + (\eta+\eps) |E\setminus B(0,R)|^{\frac{d-1}{d}} .
\end{align*}
Now, the isoperimetric inequality \eqref{eq:isopIneq} and sub-additivity of the perimeter \eqref{eq:unionint} lead to
\[|E \setminus B(0,R)|^{\frac{d-1}{d}} \ls \frac{1}{\Theta_d}\per (E\setminus B(0,R)) \ls \frac{1}{\Theta_d} \big( \per(E) + \per(B(0,R)) \big),\]
which when used in the previous equation, since $\eps$ is arbitrary and $\eta < \Theta_d$, implies that $ \per(E)$ is bounded uniformly in $\alpha$. Once again using the isoperimetric inequality yields the boundedness of $|E|$ independently of $\alpha$, as long as \eqref{eq:paramChoice} is satisfied.

We now prove that the mass and perimeter of level-sets of $u_{\alpha,w}$ are bounded away from zero. The equiintegrability of $(v_{\alpha,0})$ ensures that there is no concentration of mass for $v_{\alpha,0}$, that is, for any $\eps>0$ we can ensure $\int_E |v_{\alpha,0}|^{d} \ls \eps$ if $|E|$ small enough. Then, if $E$ belongs to $\mathcal E$, H\"older inequality provides an inequality of the type
\[\per(E) \ls \eps^{\frac{1}{d}} |E|^{\frac{d-1}{d}},\]
which together with the isoperimetric inequality, implies $\per(E) \ls \eps^{1/d} \Theta_d^{-1} \per(E)$, which is not possible for $\eps$ too small. Therefore, $|E|$ must be bounded away from zero (and $\per(E)$ as well thanks to the isoperimetric inequality).

Similarly to what is done in \cite{ChaDuvPeyPoo17,IglMerSch18}, one can (see \cite[Thm.\ B.29]{AndCasMaz04}) decompose any $E \in \mathcal E$ into  
\[E = \bigcup_{i \in I} E_i, \qquad \per(E) = \sum_i \per(E_i)\] with $I$ being finite or countable and all the $E_i$ being indecomposable (``connected''). By splitting the equation defining $E$ into the $E_i$ (similarly to \cite[Remark 4]{ChaDuvPeyPoo17}), one infers that each $E_i$ satisfies
 \begin{equation} \label{eq:ccurvei} \per(E_i) = \int_{E_i} v_{\alpha,w}\end{equation}
which implies, using the same reasoning as before, that both $|E_i|$ and $\per(E_i)$ are bounded away from zero, by constant that does not depend on $i,w$ nor $\alpha$.

We show now that each $E_i$ must have a bounded diameter. This step will actually make use of the density estimates we showed in Theorem \ref{thm:densEsts} (the proof of which does not make use of the compact support of $u_{\alpha,w}$). If it were not the case, there would exist a sequence of points $x_n \in \partial E_i \to \infty$ such that $|x_j - x_k| \gs 2r_0$ for $j\neq k$, where $r_0$ is defined in Theorem \ref{thm:densEsts}. Using the same theorem, one obtains that $|B(x_j,r_0) \cap E_i| \gs C r_0^d.$ Summing over the (disjoint) balls $B(x_j,r_0)$, we get a contradiction with the boundedness of $|E_i|$. Since both $r_0$ and $C$ are uniform, the bound on the diameter is actually independent from $i,w,\alpha$.

Finally, since $(v_{\alpha,w})$ is equiintegrable in $L^{d}$, one can find $\tilde R>0$ such that
\[\Vert v_{\alpha,w} \Vert_{L^{d}(\R^n \setminus B(0,\tilde R))} \ls \frac{\Theta_d}{2}.\]
Then, if one component $E_i$ had an empty intersection with $B(0,\tilde R)$, we would have, using \eqref{eq:ccurvei} and H\"older inequality
\[\per(E_i) \ls \frac{\Theta_d}{2} | E_i|^{\frac{d-1}{d}}.\]
We can then make use of the isoperimetric inequality to obtain
\[\per(E_i) \ls \frac{\Theta_d}{2} | E_i|^{\frac{d-1}{d}} \ls \frac{\Theta_d}{2} \frac{\per(E_i)}{\Theta_d},\]
a contradiction. Then, any component of $E$ intersects $B(0,\tilde R)$ and the bound on their diameter implies the existence of $R$ such that that $E \subset B(0,R)$.
\end{proof}

This lemma actually provides (for the particular value of $q$ assumed) the common compact support that was assumed in Theorem \ref{thm:hausConv}. This assumption can therefore be removed from that result, and we obtain:
\begin{theorem}
Assume $q=d/(d-1)$, and let $A$, $f$, and $\alpha_n, w_n \to 0$ as in Theorem \ref{thm:hausConv}, with $u_n$ minimizing $(P_{\alpha_n,w_n})$. Then, for almost every $s \in \R$, as $n$ grows to infinity, the level-sets $\{u_n>s\}$ converge to $\{u^\dag>s\}$ in the sense of Hausdorff convergence.
\end{theorem}

\paragraph*{An elementary proof of Lemma \ref{lem:farness}.}
Arguing as in the proof of Lemma \ref{lem:compsupp}, it is possible to obtain a fairly elementary proof of Lemma \ref{lem:farness} that doesn't require strong regularity results. To see this, consider 
\[ E_0 \in \underset{F}{\arg\min} \, \per(F)-\int_F f\text{, with }f=c_1 1_{B(0,r_1)}+c_2 1_{B(x_0,r_2)},\]
and assume $c_1 > 3/r_1$, $c_2 > 3/r_2$, and $E_0$ connected. 

For simplicity, we denote $B(0,r_1)$ by $B_1$ and $B(x_0,r_2)$ by $B_2$ in the rest of this argument. First, we notice that the $L^1$-optimal variational curvature $\kappa$ of $B_1 \cup B_2$ of Proposition \ref{prop:CurvL1} satisfies 
\[\kappa \, 1_{B_1 \cup B_2}  = \frac{3}{r_1} 1_{B_1}+\frac{3}{r_2} 1_{B_2}.\]
Therefore, even without knowledge of $\kappa$ outside $B_1 \cup B_2$ we can write
\begin{equation}\label{eq:compunballs}
\begin{aligned}
\per(B_1 \cup B_2)-\int_{B_1 \cup B_2}\kappa &\ls \per((B_1 \cup B_2) \cap E_0)- \int_{(B_1 \cup B_2) \cap E_0}\kappa\\
\per(E_0)-\int_{E_0}f &\ls \per((B_1 \cup B_2) \cup E_0)- \int_{(B_1 \cup B_2) \cup E_0}f,
\end{aligned}
\end{equation}
where summing and using the union-intersection inequality (as in \eqref{eq:unionint}) we get
\[0 \gs \int_{(B_1 \cup B_2) \setminus E_0}f - \kappa = \left(c_1-\frac{3}{r_1}\right)|B_1 \setminus E_0|+\left(c_2-\frac{3}{r_2}\right)|B_2 \setminus E_0|,\]
which, as $c_i > 3/r_i$, implies $B_1 \subset E_0$ and $B_2 \subset E_0$. Since $\supp f = \overline{B_1} \cup \overline{B_2}$, these inclusions mean that we can reformulate the problem for $E_0$ as minimizing perimeter with an inclusion constraint (as in the obstacle problem of \cite[Lemma, p.132]{BarMas82} and \cite[Lemma 9]{IglMerSch18}), so that
\[E_0 \in \underset{F\supseteq B_1 \cup B_2}{\arg\min} \, \per(F).\]
Now, this variational problem and the isoperimetric inequality provide us with the bound 
\begin{equation}\label{eq:massboundconst}|E_0| \ls \Theta_d^{-\frac{d}{d-1}} \per(E_0)^{\frac{d}{d-1}} \ls \Theta_d^{-\frac{d}{d-1}} \big(\per(B_1) + \per(B_2)\big)^{\frac{d}{d-1}}.\end{equation}
On the other hand, for a point $x \in \partial E_0$ and a radius $r$ such that $\Br \cap B_1 = \emptyset$ and $B(x,r) \cap B_2 = \emptyset$ we have that $\restr{f}{\Br}=0$, implying the density estimate 
\[|E_0 \cap B(x,r')| \gs C |B(x,r')|\text{ for }r'\ls r_\eps < r,\]
the constant $C$ and the maximal radius $r_{\eps}$ at which the estimate holds being independent of $x$, of $c_1, c_2$ and of the separation $|x_0|$ between the centers of $B_1$ and $B_2$. As in Lemma \ref{lem:compsupp}, if $|x_0|$ is large we may use many disjoint balls (connectedness and the fact that $E_0$ intersects $B_1$ and $B_2$ imply that we can find at least $(|x_0|-r_1-r_2)/r$ of them) to obtain a contradiction with the mass bound \eqref{eq:massboundconst}.

\subsection{The subcritical case}
If $q < d/(d-1)$, unless we work in a bounded set (see Section \ref{sec:bounded}) there is no hope to obtain a consistent regularization scheme with Hausdorff convergence of level-sets, since the data term fails to control the behavior at infinity of the solutions and subgradients. This is already hinted at in Proposition \ref{prop:linearexistence}, where we cannot guarantee obtaining a minimizer in $L^q(\R^d)$. 

To demonstrate further, we have a closer look at the two-dimensional Radon transform in $\R^2$ with measurements in $L^2$. We construct a sequence of perturbations and regularization parameters which satisfy the parameter choice inequality \eqref{eq:paramChoice}, but nevertheless force the level-sets of potential solutions to escape to infinity. The implication is that in this setting it is advisable to work in a bounded domain. For example, in \cite{BurMulPapSch14} a model is presented, which uses total variation regularization and a $L^2$ data term as an approximation of the Poisson noise model for photon emission tomography (PET) reconstruction. For this model, the analysis performed is indeed done on bounded domains.

We will need the following lemma:

\begin{lemma}\label{lem:curvannul}
Let $r_1 < r_2$, $B_1 = B(0,r_1)$, $B_2 = B(0,r_2)$ and $\mathcal A_{r_1,r_2} := B_2 \setminus B_1$ be an annulus in $\R^2.$ We denote by $\kappa_{\mathcal A}$ an optimal curvature in the sense of Proposition \ref{prop:CurvL1}. Then, $\kappa_{\mathcal A}$ is constant on $\mathcal A_{r_1,r_2}$, with value $2(r_1+r_2)/(r_2^2-r_1^2)$.
\end{lemma}
\begin{proof}
 Thanks to the rotational invariance of the problem, there exists a minimizer of 
 \[ F \mapsto \per(F) - \lambda |F| \]
among $F \subseteq \mathcal A_{r_1,r_2}$ which is rotationally invariant. We can furthermore decompose it into ``connected components'' as in Lemma \ref{lem:farness}: there is a minimizer which is an annulus $\mathcal A_{r_a, r_b}.$ Computing the energy of this annulus makes clear that such a minimizer is actually either empty (if $\lambda \ls 2(r_1+r_2)/(r_2^2-r_1^2)$) or equal to $\mathcal A_{r_1,r_2}$ (if $\lambda \gs 2(r_1+r_2)/(r_2^2-r_1^2)$).
\end{proof}

\begin{example}\label{ex:radoncounterex}
We consider $A=\Rd$ the Radon transform in the plane, $Y=L^2([0,2\pi) \times \R^+)$, $\sigma=2$ and $q=4/3$. We note that in the plane $\Rd$ is a bounded operator from $L^{4/3}$ to $L^2$ \cite{Obe83, ObeSte82}. The starting point is the noiseless measurement $f:=\Rd 1_{B(0,1)}$, which since $\Rd$ is injective, gives rise to the corresponding minimal variation solution $u^\dag = 1_{B(0,1)}$. 
For a fixed $\delta \in (0,1/2)$, we define the perturbation
\begin{equation}w_n := \Rd z_n\text{, for }z_n :=\frac{1}{n^{3/2+\delta}}1_{\A_n}\text{ with }\A_n := B(0,2n) \setminus B(0,n).\end{equation}
The corresponding sequence of regularization parameters is defined as $\alpha_n=\ell n^{-\delta} \to 0$ for a constant $\ell$, for which we can compute
\[\frac{1}{\alpha_n}\|w_n\|_{Y} \ls C\frac{1}{\alpha_n}\|z_n\|_{L^{4/3}(\R^d)}=C \alpha_n^{-1} n^{-3/2-\delta} |\A_n|^{3/4} = C \ell^{-1} n^{-3/2} |\A_n|^{3/4}\ls C \ell^{-1},\]
meaning that the parameter choice inequality required for Hausdorff convergence of level-sets \eqref{eq:paramChoice} holds if $\ell$ is chosen large enough. Notice also that the condition listed in Proposition \ref{prop:linearexistence} for convergence of minimizers is also automatically satisfied, since in addition to the above we have $\|w_n\|_{Y} \to 0$.

Now, assume for the sake of contradiction that we had a sequence $u_n$ of minimizers of \eqref{eq:primalalphaw}, all of them supported in a compact set $B$. Using $\Rd u^\dag = f$, the optimality condition for \eqref{eq:primalalphaw} reads
\[v_n := -\frac{1}{\alpha_n} \Rda \Rd \big(u_n - u^\dag - z_n\big) \in \partial \TV{u_n}.\]
However \cite[Thm.\ 1.5]{Nat01}, the operator $\Rda \Rd$ is proportional the Riesz potential operator of order one,
\[\Rda \Rd u = 2\, I_1 u \text{, with }I_1 u(x):=\int_{\R^2}\frac{u(y)}{|x-y|}\dd y,\]
which allows us to consider $x \in \A_n$ (for which we have $u^\dag(x) = 0$) and estimate $v_n(x)$ for large $n$. On the one hand we have the common compact support $B$ for all $u_n$, implying that
\begin{equation}\label{eq:ucontrib}\begin{aligned}\frac{1}{\alpha_n}I_1 u_n(x) &= \frac{1}{\alpha_n}\int_{\R^2}\frac{u_n(y)}{|x-y|}\dd y \ls \frac{1}{\alpha_n d(x, \supp u_n)}\|u_n\|_{L^1(\R^d)} \\ &\ls \frac{1}{\alpha_n d(x, B)}\|u_n\|_{L^1(\R^d)} \ls C \ell^{-1} n^{-1+\delta},\end{aligned}\end{equation}
where we have used $x \in \A_n$, the common compact support and Proposition \ref{prop:linearexistence} to conclude that $\|u_n\|_{L^1(\R^d)}$ is a bounded sequence, since $u_n \to u^\dag$ in $L^1(\R^d)$. Notice that \eqref{eq:ucontrib} also holds when replacing $u_n$ by $u^\dag$. On the other hand, we have
\begin{equation}\label{eq:zcontrib}\begin{aligned}\frac{1}{\alpha_n} I_1 z_n(x)&=\frac{1}{\alpha_n}\int_{\R^2}\frac{z_n(y)}{|x-y|}\dd y=\frac{1}{\alpha_n}\int_{\A_n}\frac{z_n(y)}{|x-y|}\dd y \gs \frac{1}{\alpha_n \diam(\A_n)}\int_{\A_n}z_n(y)\dd y \\
&\gs \frac{C}{n \alpha_n}\int_{\A_n}z_n(y)\dd y = \frac{C |\A_n| n^{-3/2-\delta}}{n \alpha_n} = C \ell^{-1} n^{-1/2},
\end{aligned}\end{equation}
from which, in combination with \eqref{eq:ucontrib} and since $\delta < 1/2$, we can conclude that in fact also $v_n(x)\gs C_{v} \ell^{-1} n^{-1/2}$ for all $x \in \A_n$, some constant $C_v$ and $n$ large enough. Additionally, by Lemma \ref{lem:curvannul} the optimal curvature $\kappa_{\A_n}$ of $\A_n$ satisfies $\kappa_{\A_n}(x)\ls C_{k}/n$ for all $x$ in $\A_n$ and some constant $C_k$. Combining these two estimates we have therefore for $n$ large enough the pointwise curvature comparison 
\[v_n(x) - \kappa_{\A_n}(x) \gs \big(C_v \ell^{-1} - C_k n^{-1/2}\big) n^{-1/2} >  \frac12 C_v \ell^{-1}n^{-1/2} > 0,\]
which in turn (comparing as in \eqref{eq:compunballs}) implies that any minimizer of
\[F\mapsto \per(F) - \int_F v_n\]
must contain $\A_n$, a contradiction with the common compact support for all $u_n$.
\end{example}
\section{Remarks on bounded domains and boundary conditions}\label{sec:bounded}
So far, the convergence results that we have proved hold for functions defined in the whole $\R^d$. Nevertheless, our results also apply to bounded domains with either Dirichlet or Neumann boundary conditions on a bounded set $\Omega$ that satisfies mild regularity assumptions. This adaptation has been explained in detail for solutions in $L^2(\R^2)$ in \cite{IglMerSch18}, and as we have seen, restricting to bounded domains is also necessary for the case $q < d/(d-1)$. We now briefly present the required constructions. 
 
\paragraph{Bounded domain with Dirichlet conditions.} Here, $u \in L^q(\Omega)$ and $\TV{u}$ is the total variation, computed in $\R^d$, of the extension $\tilde u$ of $u$ by zero outside $\Omega$. Differently said, it means that the jump of $u$ to zero at the boundary of $\Omega$ is taken into account. This is well defined if $\Omega$ is an extension domain, for example Lipschitz. In the following, we will need that $\Omega$ has also a variational curvature $\kappa_\Omega$ that satisfies
\begin{equation} \label{eq:DomDirich} \restr{\kappa_\Omega}{\R^d \setminus \Omega} \in L^d(\R^d \setminus \Omega). \end{equation}
In particular, any convex or $\mathcal C^{1,1}$ domain will satisfy this assumption. 

The existence of a minimizer for the approximate problem and its convergence in $L^1$ to a minimal total variation solution (Proposition \ref{prop:linearexistence}) as well as the duality analysis (Propositions \ref{prop:lineardual} and \ref{prop:dualconv}) still hold with no modification. The results related to the parameter choice (Propositions \ref{prop:shiftingV} and \ref{prop:genproj}) depend only on the space $Y$ and the dimension and therefore are not affected by changing the boundary conditions.

However, the proof of density estimates has to be slightly modified, since it is not possible anymore to consider, for a level-set $E$ of $u_{\alpha,w}$, the competitor $E \cup B_r$. Indeed, such a set $E$ would minimize
\[ F \mapsto \per(F) - \operatorname{sign}(s) \int_F v_{\alpha,w} \]
only among the subsets $F$ of $\Omega$. The strategy is then to relax the constraint $F \subset \Omega$ and introduce 
\[\kappa_{\alpha,w} := \operatorname{sign}(t)v_{\alpha,w} 1_{\Omega} + \kappa_\Omega 1_{\R^d \setminus \Omega} \]
where $\kappa_\Omega$ is a variational curvature for $\Omega.$ One can then show (see \cite[Lemma 9]{IglMerSch18}) that $E$ minimizes
\[ F \mapsto \per(F) - \int_F \kappa_{\alpha,w} \]
among $F \subset \R^d$ (without inclusion constraint). Then, provided \eqref{eq:DomDirich}, density estimates are obtained as before, since the functions $(\kappa_{\alpha,w})$ are also equiintegrable in $\R^d$. 

\paragraph{Bounded domain with Neumann boundary conditions.} In this case, $u \in L^q(\Omega)$ and one uses $\TVO{u}$, the total variation computed in $\Omega$ (the jumps at the boundary of $\Omega$ are \emph{not} taken into account). In this case, the proof of the existence result (Proposition \ref{prop:linearexistence}) needs to take into account the behaviour of the operator $A$ on constant functions, in exactly the same way as done in \cite[Prop.\ 2]{IglMerSch18}. Other than that, everything until Proposition \ref{prop:genproj} works similarly. Proposition \ref{prop:subgradTV} then implies that any level-set $E$ of a minimizer $u_{\alpha,w}$ minimizes among $F \subset \Omega$
$$ F\mapsto \per(F ; \Omega) - \int_F v_{\alpha,w},$$
where $\per(F ; \Omega) := \TVO{1_F}$ is the perimeter in $\Omega$, defined as in \eqref{eq:defTV}, but with test functions in $\mathcal C_0^\infty(\Omega;\R^d)$. For this relative perimeter, the standard isoperimetric inequality does not hold. To see this, consider for example that if $x \in \partial \Omega$ and $r \to 0^+$, then $\per(\Omega \setminus B(x,r);\Omega)\to 0$ while $|\Omega \setminus B(x,r)| \to |\Omega|$. Nevertheless, provided $\Omega$ is Lipschitz, the Sobolev inequality \cite[Remark 3.50]{AmbFusPal00} writes for $u\in \BV(\Omega)$,
\begin{equation}\label{eq:sobolevineq}\norm{u - \frac{1}{|\Omega|}\int_\Omega u}_{L^{\frac{d}{d-1}}(\Omega)} \ls C_\Omega \,\TV{u\, ;\, \Omega}.\end{equation} Taking $u = 1_F$ for any $F \subset \Omega$, we obtain \cite[Section 4.3]{IglMerSch18}
\[ C_\Omega \per(F ; \Omega) \gs \frac{|F|^{\frac{d}{d-1}} |\Omega \setminus F|^{\frac{d}{d-1}}}{|\Omega|^{\frac{d}{d-1}}}. \]
which can play the role of the isoperimetric inequality in the proof of density estimates. Note that now the parameter choice \eqref{eq:paramChoice} has to be made relatively to the constant $C_\Omega,$ that is $\eta < C_\Omega$ in \eqref{eq:paramChoice}.

\paragraph{Periodic boundary conditions.} It is also possible to treat the case of periodic boundary conditions, commonly used in image processing (see for example \cite[Sec.\ 3.3]{AubKor06}). A reasonable definition of periodic total variation is given in \cite{CesNov13}, which we now describe using their same notation. Let $Q=(0,1)^d$ be the $d$-dimensional cube. For $u \in \BV(Q)$ we denote by $u_{\partial Q} \in L^1(\partial Q)$ its trace on $\partial Q$ (which exists by \cite[Thm.\ 3.87]{AmbFusPal00}). Moreover, we define the part of the boundary
\[\partial_0 Q := \partial Q \cap \left\{ x = (x_1, \ldots \, x_d) \,\middle\vert\, \textstyle{\prod\limits_{i=1}^d} x_i = 0\right\},\]
where jumps should be accounted for in the variation. To accomplish this, one can use the boundary map $\zeta: \partial_0 Q \to \partial Q$ defined by
\[\zeta(x)=x+\sum_{i=1}^d \gamma_i(x) e_i, \text{ for }\gamma_i(x)=1 \text{ if }x_i = 0 \text{ and }\gamma_i(x)=0\text{ otherwise.}\]
With it, one can define the periodic total variation of $u \in \BV(Q)$ to be
\begin{equation}\label{eq:tvper}\TVper{u\, ;\, Q}:=\TV{u\, ;\, Q}+\int_{\partial_0 Q} \big|u_{\partial Q}(x) - u_{\partial Q}\big(\zeta(x)\big)\big|\dd \mathcal{H}^{d-1}(x),\end{equation}
and the corresponding perimeter of a set $E \subset Q$ as $\TVper{1_E\, ;\, Q}$. With these definitions, we can just notice 
\[\TVper{u\, ;\, Q}  \gs \TV{u\, ;\, Q},\]
so that the Sobolev inequality \eqref{eq:sobolevineq} remains valid with the same constant, allowing us to proceed as in the Neumann case for the proof of the density estimates. Existence is likewise treated as in the Neumann case, and all the other results hold with no modification. Notice in particular that even if the expression \eqref{eq:tvper} of $\mathrm{TV}_{\mathrm{per}}$ contains boundary terms, periodicity implies that the difficulties that make the Dirichlet case require extensions do not arise.

\section*{Acknowledgements}
We would like to thank Otmar Scherzer for encouragement to work on the interplay between the space dimension and convergence of level-sets.

\bibliographystyle{plain}
\bibliography{conv_general_source_cond}
\end{document}